\documentclass[leqno,a4paper,11pt]{amsart}
\usepackage[a4paper=true,pdfpagelabels]{hyperref}
\usepackage{graphicx}

\usepackage[ansinew]{inputenc}
\usepackage{amsfonts,epsfig}
\usepackage{latexsym}
\usepackage{mathabx}
\usepackage{amssymb}
\usepackage{amsmath}

\usepackage{color}
\usepackage{mathrsfs} 
\usepackage{graphicx}
\usepackage{enumitem}

\newcommand{\vertiii}[1]{{\left\vert\kern-0.25ex\left\vert\kern-0.25ex\left\vert #1 
    \right\vert\kern-0.25ex\right\vert\kern-0.25ex\right\vert}}

\newcommand{\D}{\mathbb{D}}

\newcommand{\T}{\mathbb{T}}
\newcommand{\N}{\mathbb{N}}

\newcommand{\C}{\mathbb{C}}

\newcommand{\R}{\mathbb{R}}

\renewcommand{\H}{\mathcal{H}}
\renewcommand{\Re}{\operatorname{Re}}

\newcommand{\A}{\mathscr{A}}
\newcommand{\B}{\mathscr{B}}
\newcommand{\BB}{\mathcal{B}}

\newcommand{\KK}{\mathcal{K}}

\def\la{\lambda}           
              
                  \def\z{\zeta}

\newcommand{\al}{\alpha}
\newcommand{\ga}{\gamma}

\newtheorem{theorem}{Theorem}[section]
\newtheorem{lemma}[theorem]{Lemma}
\newtheorem{corollary}[theorem]{Corollary}
\newtheorem{proposition}[theorem]{Proposition}
\newtheorem{question}[theorem]{Question}

\theoremstyle{definition}
\newtheorem{definition}[theorem]{Definition}

\newtheorem{remark}[theorem]{Remark}

\numberwithin{equation}{section}

\theoremstyle{definition}

\numberwithin{equation}{section}

\theoremstyle{plain} 
\newcommand{\thistheoremname}{}
\newtheorem*{genericthm*}{\thistheoremname}
\newenvironment{namedthm*}[1]
{\renewcommand{\thistheoremname}{#1}%
	\begin{genericthm*}}
	{\end{genericthm*}}

\makeatletter
\@namedef{subjclassname@2020}{\textup{2020} Mathematics Subject Classification}
\makeatother

\begin{document}

\title
{Composition of analytic paraproducts}

\author[A. Aleman]{Alexandru Aleman}
\address{A. Aleman: Department of Mathematics, University of Lund, P.O.\ Box 118, SE-221 00, Lund, Sweden}
\email{alexandru.aleman@math.lu.se}
\author[C. Cascante]{Carme Cascante}
\address{C. Cascante: Departament de Matem\`atiques i
    Inform\`atica, Universitat  de Barcelona,
     Gran Via 585, 08071 Barcelona, Spain}
\email{cascante@ub.edu}
\author[J. F\`abrega]{Joan F\`abrega}
\address{J. F\`abrega: Departament de Matem\`atiques i
    Inform\`atica, Universitat  de Barcelona,
     Gran Via 585, 08071 Barcelona, Spain}
\email{joan$_{-}$fabrega@ub.edu}
\author[D. Pascuas]{Daniel Pascuas}
\address{D. Pascuas: Departament de Matem\`atiques i
    Inform\`atica, Universitat  de Barcelona,
     Gran Via 585, 08071 Barcelona, Spain}
\email{daniel$_{-}$pascuas@ub.edu}
\author[J. A. Pel\'aez]{Jos\'e \'Angel Pel\'aez}
\address{J. A. Pel\'aez: Departamento de An\'alisis Matem\'atico, Universidad de M\'alaga, Campus de
Teatinos, 29071 M\'alaga, Spain} \email{japelaez@uma.es}

\thanks{
The research of the second, third and fourth author
	was supported in part by
        Ministerio de Econom\'{\i}a y Competitividad, 
Spain,   project  
MTM2017-83499-P,  
and Generalitat de Catalunya, 
project
2017SGR358.
The research of the fifth author was supported in part by Ministerio de Econom\'{\i}a y Competitividad, Spain, projects
PGC2018-096166-B-100; La Junta de Andaluc{\'i}a,
projects FQM210  and UMA18-FEDERJA-002}
\date{\today}



\begin{abstract} 
For a fixed  analytic function $g$ on the unit disc $\D$, we consider the analytic paraproducts induced by $g$, which are defined by 
$T_gf(z)= \int_0^z f(\z)g'(\z)\,d\z$,
$S_gf(z)= \int_0^z f'(\z)g(\z)\,d\z$, and
$M_gf(z)= f(z)g(z)$.
  The boundedness of these operators on various spaces of analytic functions on $\D$ is well understood. The original motivation for this work is to understand the boundedness of compositions of two of these operators, for example $T_g^2, \,T_gS_g,\,  M_gT_g$, etc. Our methods yield a characterization of the  boundedness    of a large class of operators contained in  the algebra 
generated by these  analytic paraproducts acting on the classical weighted Bergman and Hardy spaces in terms of the symbol $g$. In some cases it turns out that this property is not  affected by cancellation, while in others it requires stronger and more subtle restrictions on the oscillation of the symbol $g$ than  the case of a single paraproduct. 
\end{abstract}
\maketitle


\section{Introduction}

Let $\H(\D)$ be the space of analytic functions on the unit disc $\D$ of the complex plane. For  $\alpha>-1$ and $0<p<\infty$, the weighted Bergman space
 $A^p_{\alpha}$ consists of the functions $f\in \H(\D)$ such that
\[
\|f\|_{\alpha,p}^p:=
(\alpha+1)\int_{\D}|f(z)|^p(1-|z|^2)^{\alpha}\,dA(z)<\infty,
\] 
where $dA$ is the the normalized area measure on $\D$.
Let $H^p$, $0<p\le \infty$, denote the classical Hardy space of analytic functions in $\D$.
To simplify the  notations,  we shall write $A^p_{-1}:=H^p$ and $\|\cdot\|_{-1,p}:=\|\cdot\|_{H^p}$, $0<p<\infty$. Given $g\in \H(\D)$, let us consider the multiplication operator $M_gf= fg$ and the integral operators
\[
\begin{aligned}
T_gf(z)= \int_0^z f(\z)g'(\z)\,d\z \quad 
S_gf(z)= \int_0^z f'(\z)g(\z)\,d\z\qquad
(z\in\D).
\end{aligned}
\]
Due to the integration by parts relation
\begin{equation}\label{eqn:formula}
 M_gf=T_gf+S_gf +g(0)f(0),
 \end{equation}
 we call these operators {\em{analytic paraproducts}}.
 
It is well-known  \cite{Aleman:Cima,Aleman:Constantin,Aleman:Siskakis:1,Pommerenke}  that  $T_g$ is bounded  on $A^p_\alpha$ if and only if $g$ belongs to 
the Bloch space $\B$  when $\alpha>-1$,  and $g\in BMOA$  in the Hardy space case $\alpha=-1$. In particular, these results show that cancellation plays a key role 
in the boundedness of the integral operator $T_g$. This is very different from the case of $M_g$ and $S_g$, whose boundedness on these spaces is equivalent to the boundedness of $g$ in $\D$ (see  Proposition \ref{prop:SM-bounded-compact}  below and the references following it). 

Throughout the paper
the spaces of bounded and compact linear operators on $A^p_{\alpha}$ are denoted by $\BB(A^p_{\alpha})$ and $\KK(A^p_{\alpha})$, respectively.
 Moreover,  if  $T:A_\al^p\to A_\al^p$ is a linear map, we write 
$\|T\|_{\alpha,p}=\sup_{\|f\|_{\al,p}\le 1}\|Tf\|_{\al,p}$,
and we refer to this quantity as the operator norm of $T$ on $A^p_{\alpha}$, despite that $A^p_{\alpha}$ is not a normed space for~{$0<p<1$.}

The primary aim of this paper is to study the boundedness of compositions (products) of analytic paraproducts acting on $A^p_\alpha$. In order to provide some intuition and motivation for this circle of problems, let us have a brief look at compositions of two such paraproducts. Clearly, $M_g^2=M_gM_g$ is bounded on these spaces if and only if $g\in H^\infty$ and we shall show (Theorem \ref{thm:main:introduction}) that the same holds for $S_g^2$. On the other hand, it turns out that $T_g^2\in \BB(A^p_{\alpha})$  if and only if $T_g\in \BB(A^p_{\alpha})$  (Theorem \ref{thm:algebraimpliesginBloch:introduction}). Regarding mixed products, a simple computation reveals that  
$S_gT_g=T_gM_g=\frac{1}{2}T_{g^2}$, so that
both compositions are bounded on $A^p_\alpha$  if and only if $g^2\in \B$, when $\alpha>-1$, or $g^2\in BMOA$, when $\alpha=-1$. This last condition  is strictly stronger than $g\in \B$ or $g\in BMOA$, respectively
 (see Proposition~\ref{prop:Bn:nested} below).  The compositions in reversed order raise additional problems because they  cannot be expressed as a single paraproduct. They can be related to the previous ones using~(\ref{eqn:formula}):
\begin{equation}\label{2letter}
M_gT_g= S_gT_g+T^2_g,\quad  T_gS_g=S_gT_g-T^2_g-g(0)(g-g(0))\delta_0,
\end{equation} 
where $\delta_0f=f(0)$.
Intuitively, from above it appears that $S_gT_g=\frac{1}{2}T_{g^2}$ is the \lq\lq dominant term\rq\rq\,  in both sums, but a  priori it is not clear whether such sums are affected by cancellation or not.  Thus we are led in a natural way to consider sums of compositions of paraproducts  rather than only compositions.   A similar situation occurs when dealing with $M_gS_g$ and  $S_gM_g$.\\
Due to these preliminary observations we turn our attention to the full algebra $\A_g$  
generated by the operators $M_g$, $S_g$, and $T_g$.
 The operators in  $\A_g$
will be called {\em $g$-operators}.  In this general  framework  we begin by showing that any $g$-operator $L$  has a representation
\begin{equation}\label{eqn:general:expression:operator}
L=\sum_{k=0}^nS_g^kT_gP_k(T_g)+S_gP_{n+1}(S_g)+g(0)P_{n+2}(g-g(0))\,\delta_0,
\end{equation}
for some $n\in\N$, where the $P_k$'s are polynomials.
This representation is essentially unique, see~{\S\ref{sec:Vector space structure}}. If $P_k=0$, for $0\le k\le n+1$, we will say that $L$ is a {\em trivial operator}.
With this representation in hand we can derive  a fairly surprising necessary condition for the boundedness of general operators in this algebra.  
\begin{theorem}\label{thm:algebraimpliesginBloch:introduction} $\mbox{ }$
Let $g\in \H(\D)$, $0<p<\infty$ and $\alpha\ge -1$.
Let $L$ be a $g$-operator written in the form \eqref{eqn:general:expression:operator}. Then:
\begin{enumerate}[label={\sf\alph*)},ref={\sf\alph*)},topsep=3pt, leftmargin=*,itemsep=3pt] 
\item \label{item:thm:algebraimpliesginBloch:introduction:1}
If $L$ is a non-trivial operator and $L\in\BB(A^p_{\alpha})$, then $T_g\in\BB(A^p_{\alpha})$, that is, 
$g\in\B$, when $\alpha>-1$,  and $g\in BMOA$, when  $\alpha=-1$. 
\item \label{item:thm:algebraimpliesginBloch:introduction:2}
If $L$ is a non-zero trivial operator, then $L\in\BB(A^p_{\alpha})$ if and only if 
$g^{deg P_{n+2}}\in A^p_\alpha$.
\end{enumerate}
\end{theorem}
Note that the result applies directly to $T_gS_g$ and $M_gT_g$ via \eqref{2letter} and  justifies the  intuition that these operators 
 are bounded on   $A^p_\alpha$  if and only if $g^2\in \B$, when $\alpha>-1$, or $g^2\in BMOA$, when $\alpha=-1$. In fact, 
 in~{\S\ref{2letter-words}}  we provide a complete characterization of the boundedness of compositions of two analytic paraproducts (see also Corollary \ref{prop:boundedness:two:letters:operators} below).
 The theorem can be used to characterize the boundedness of more complicated $g$-operators.
In addition, it provides a crucial preliminary step in the proof of our characterization of boundedness of certain $g$-operators which we now state.
\begin{theorem}\label{thm:main:introduction}
Let $g\in \H(\D)$, $0<p<\infty$ and $\alpha\ge -1$. 
Let $L$ be a $g$-operator written in the form \eqref{eqn:general:expression:operator}. Then: 
\begin{enumerate}[label={\sf\alph*)},ref={\sf\alph*)},
topsep=3pt, leftmargin=*,itemsep=3pt] 
\item\label{item:thm:main:introduction:1}
If $P_{n+1}\ne 0$,
$L\in \BB(A^p_{\alpha})$ if and only if 
$g\in H^\infty$.
\item\label{item:thm:main:introduction:2} 
If $P_{n+1}=0$ and $P_n=1$, $L\in\BB(A_\al^p)$ if and only if  $T_{g^{n+1}}\in\BB(A_\al^p)$, or equivalently,  
$g^{n+1}\in\B$, when $\al>-1$, and $g^{n+1}\in BMOA$, when $\al=-1$.
\item\label{item:thm:main:introduction:3} 
If $\alpha>-1$, $P_{n+1}=0$, and  $P_n(0)\ne 0$, $L\in \BB(A^p_{\alpha})$ if and only if $g^{n+1}\in \B$.
 \end{enumerate}
\end{theorem} 
We have not been able to extend part c) of this theorem to  the  $H^p$-case. One direction follows directly from  Proposition~\ref{prop:Bn:nested}, but the remaining one is, in  our opinion, an interesting and challenging open question.
\begin{question}\label{the question} Let  $g\in \H(\D)$, $0<p<\infty$,  and let $L$ be a $g$-operator written in the form \eqref{eqn:general:expression:operator} with $P_{n+1}=0$, and  $P_n(0)\ne 0$, which is  bounded on $H^p$. Is it true that  $g^{n+1}\in BMOA$? \end{question}

When dealing with operators in  $\A_g$, an initial hurdle can be easily recognized, namely that these operators  are formally defined as sums of products of possibly unbounded operators on the spaces in question.  One way to overcome this difficulty is to consider dilations of the symbol $g$, which are defined, for $\lambda\in\overline{\D}$, by $g_\lambda(z)=g(\lambda z)$. In Proposition~\ref{prop:norm:dilations} we prove that if  
$L_g \in \A_g\cap \BB(A^p_{\alpha})$ then $\|L_{g_{\lambda}}\|_{\alpha,p}\le \|L_g\|_{\alpha,p}$ for all $\lambda\in\overline{\D}$. This fact  will be repeatedly used 
in the proofs of the results stated above. Other key ingredients for the proof of Theorem \ref{thm:algebraimpliesginBloch:introduction} are the estimates 
\[
\|T_g\|_{\alpha,p}^n\le c_n\|T_g^n\|_{\alpha,p},
\]
which will be established in 
 Proposition~\ref{prop:T_g-powers},
together with the analysis of iterated commutators of $T_g$ and $S^k_g$, $k\in\N$. A sample of this set of ideas can be found in Corollary~\ref{co:iteratedcom} below. The proof of Theorem \ref{thm:main:introduction} is somewhat more involved, in particular, it makes use of the boundary behaviour of $A_\alpha^p$-valued functions of the form $\lambda\to L_{g_\lambda}f$, $\lambda\in \D$,  
$f\in A_\alpha^p$.
\medskip

In order to discuss the class of $g$-operators covered by 
Theorem~{\ref{thm:main:introduction} \ref{item:thm:main:introduction:2}} it is convenient to introduce  the following terminology. An {\em $n$-letter $g$-word} is a   $g$-operator of the form $L=L_1\cdots L_n$, where each $L_j$ is either $M_g$, $S_g$ or $T_g$.  For $n\in\N$, let $\A^{(n)}_g$ be the  linear span of $g$-words with no more than $n$ letters and define 
the {\em order} of a $g$-operator $L$ to be the least $n\in\N$
such that $L\in\A^{(n)}_g$. It turns out that  if $L\in\A^{(n)}_g$ then  the words involved in its
representation~{\eqref{eqn:general:expression:operator}} have length at most $n$. For example,  $g$-operators of order two have the form
\begin{equation}\label{order2}
L=a_1T_g+a_2T_g^2+a_3 S_g T_g+a_4 S_g+a_5 S_g^2+ g(0)P(g-g(0))\delta_0,
\end{equation}
where  $a_j\in\C,~1\le j\le 5$, and $P$ is a polynomial of degree  smaller than $2$. These operators are covered by  Theorem~\ref{thm:main:introduction}, and  we have the following complete characterization of their boundedness.

\begin{corollary}\label{prop:boundedness:two:letters:operators}
Let $g\in \H(\D)$, $0<p<\infty$ and $\alpha\ge -1$. If $L$ is a $g$-operator of order two written in the form \eqref{order2}, then:
\begin{enumerate}[label={\sf\alph*)},ref={\sf\alph*)},
topsep=3pt,leftmargin=*,itemsep=3pt] 
\item When either $a_4\ne0$ or $a_5 \ne0$, $L\in\BB(A^p_{\alpha})$ if and only if $g\in H^\infty$.
\item When $a_3\ne 0$ and $a_4=a_5=0$,  $L\in\BB(A^p_{\alpha})$  if and only if  $g^2\in \B$, for $\alpha>-1$, and $g^2\in BMOA$, for $\alpha=-1$.
\item When $a_3=a_4=a_5=0$ and either $a_1\ne0$ or $a_2\ne0$,  $L\in\BB(A^p_{\alpha})$  if and only if $g\in \B$, for $\alpha>-1$, and $g\in BMOA$, for $\alpha=-1$.
\end{enumerate}
\end{corollary}

On the other hand, our main result does not cover $g$-operators  with $P_{n+1}=0$ and $P_n(0)=0$ in the representation \eqref{eqn:general:expression:operator}. An  example of this type,  where the condition for the boundedness  is different,  follows from the second identity in~{\eqref{2letter}}. This together with $S_gT_g=\tfrac12T_{g^2}$  implies
\[
\tfrac14T_{g^2}^2=S_g(T_gS_g)T_g=S_g^2T_g^2-S_gT_g^3,
\]
{\em i.e.}\ the operator on the right  is the representation \eqref{eqn:general:expression:operator} of $\tfrac14T_{g^2}^2$.   In view of Theorem \ref{thm:main:introduction}  one might expect  that the presence of $S_g^2$ forces the boundedness of $T_{g^3}$, but by  Theorem \ref{thm:algebraimpliesginBloch:introduction} this operator is bounded on $A_\alpha^p$ if and only if  $g^2\in \B$, for $\alpha>-1$, and $g^2\in BMOA$, for $\alpha=-1$.\\ There are also $g$-operators of order $3$  with $P_{n+1}=0$ and $P_n(0)=0$ in the representation \eqref{eqn:general:expression:operator}. 
The simplest example is the $3$-letter-word $S_gT_g^2$ and in this case the situation differs even more dramatically to the one described in Theorem \ref{thm:main:introduction}. 
The following result shows that the boundedness of such $g$-operators cannot be characterized with conditions of the form $g\in H^\infty$, or $g^n\in \B\, (BMOA)$, with $n\in \N$. 

As usual, we denote by $\log$  the principal branch of the logarithm on 
$\C\setminus(-\infty,0]$, that is, $\log 1=0$. For an open set $U\subset \C$  and an analytic function   $h:U\to\C\setminus(-\infty,0],~\beta\in\C$, we define
$h^\beta=\exp(\beta\log h)$.

\begin{theorem}\label{thm:counterexamples}
Consider the function $g:\D\to\C\setminus(-\infty,0]$ defined by
\[
g(z)=\log\biggl(\frac{e}{1-z}\biggr) \qquad(z\in\D).
\]
 Then: 
\begin{enumerate}[label={\sf\alph*)},ref={\sf\alph*)},
topsep=3pt,leftmargin=*,itemsep=3pt] 
\item  $g\in BMOA$, but  for any $\alpha\ge -1,~p>0$,  we have $S_gT_g^2\notin \BB(A_\alpha^p)$.
\item For  $\frac12<\beta<\frac23$, $g^{2\beta}\notin\B$ (and so $g^{2\beta}\notin BMOA$), but  
 $S_{g^{\beta}}T^2_{g^{\beta}}\in\KK(A^p_{\alpha})$, for any $\alpha\ge-1$ and $p>0$.
\end{enumerate}
\end{theorem}

The paper is organized as follows.
 Section \ref{sec:spaces:of:symbols} contains some preliminary results concerning the Bloch space and $BMOA$, in particular  the condition  $g^k\in \B\, (BMOA)$, for some $k\in \N$.
In Section \ref{sec:algebra:of:g-operators} we   study the vector space structure of the algebra $\A_g$ and prove the representation \eqref{eqn:general:expression:operator}.
Section \ref{sec:boundedness:g-operators} is devoted to the proof of our main results, 
Theorems~\ref{thm:algebraimpliesginBloch:introduction} 
and~\ref{thm:main:introduction}. Finally, in the last section we prove Theorem~\ref{thm:counterexamples}.

As usual, $\N$ is the set of positive integers and $\T=\{z\in\C:|z|=1\}$ is the unit circle. For $\lambda\in\C$ and $r>0$, $D(\lambda,r)=\{z\in\C:|z-\lambda|<r\}$ is the open disc centered at $\lambda$ with radius $r$. For two non-negative functions $A$ and $B$, 
$A\lesssim B$ ($B\gtrsim A$) means that there is a finite positive constant $C$, independent of the variables involved, which satisfies $A\le C\,B$. Moreover, we will write $A\simeq B$ when $A\lesssim B$ and $B\lesssim A$. 

\section{The spaces of symbols}\label{sec:spaces:of:symbols}
  
In this section we will recall and prove  some preliminary results about $BMOA$ and the Bloch space.
For any $a\in\D$, define $\phi_a(z):=\frac{a-z}{1-\overline{a}z}$, and consider the classical $BMOA$ and Bloch spaces
endowed with their Garsia's seminorms $\vertiii{\cdot}_{BMOA}$ and  $\vertiii{\cdot}_{\B}$ (see, for instance, \cite{Axler, Girela} and the references therein): 
\begin{align*}
BMOA
&:=\Bigl\{f\in\H(\D)\,:\,\vertiii{f}^2_{BMOA}:=
\sup_{a\in\D}\|f\circ\phi_a-f(a)\|^2_{H^2}<\infty\Bigr\} \\
\B
&:=\Bigl\{f\in\H(\D)\,:\,\vertiii{f}^2_{\B}:=\sup_{a\in\D}\|f\circ\phi_a-f(a)\|^2_{A^2}<\infty\Bigr\}.
\end{align*}
For a given Banach space (or a complete
metric space) $X$ of analytic functions on $\D$, a positive Borel measure $\mu$ on $\D$ is called
a $q$-Carleson measure for $X$ (vanishing $q$-Carleson measure for $X$) if the identity operator $I: X \to L^q(\mu)$ is bounded (compact).
Recall that 
 $f\in\B$ if and only if $\|f\|_{\B}:=\sup_{z\in\D}(1-|z|^2)|f'(z)|<\infty$, and $f\in BMOA$ if and only if $(1-|z|^2)\,|f'(z)|^2\,dA(z)$ is a Carleson measure for $H^p$, $0<p<\infty$,
 or equivalently
\[
\|f\|^2_{BMOA}:=\sup_{a\in\D}\int_{\D}(1-|\phi_a|^2)\,|f'|^2\,dA<\infty.
\]
Moreover, $\vertiii{f}_{\B}\simeq\|f\|_{\B}$ and
$\vertiii{f}_{BMOA}\simeq\|f\|_{BMOA}$.

We also consider the little-oh subspaces of $BMOA$ and $\B$: 
\begin{align*}
VMOA
&:=\Bigl\{f\in H^2: \lim_{|a|\to 1^{-}}
\|f\circ\phi_a-f(a)\|^2_{H^2}=0\Bigr\} \\
\B_0
&:=\Bigl\{f\in A^2\,:\,\lim_{|a|\to 1^{-}}
\|f\circ\phi_a-f(a)\|^2_{A^2}=0\Bigr\}.
\end{align*}
For $f\in\H(\D)$, recall that $f\in\B_0$ if and only if 
$\lim_{|z|\to1^{-}}(1-|z|^2)|f'(z)|=0$, and $f\in VMOA$ if and only if  $(1-|z|^2)\,|f'(z)|^2\,dA(z)$ is a vanishing Carleson measure for
 $H^p$, $0<p<\infty$,
or equivalently
\[
\lim_{|a|\to1^{-}}\int_{\D}(1-|\phi_a|^2)\,|f'|^2\,dA=0.
\] 
For $0<p<\infty$ and $m,n\in\N$, $m\le n$, Jensen's inequality shows that 
$\|f^m\|^{1/m}_{\alpha,p}\le \|f^n\|^{1/n}_{\alpha,p}$.
We will show that this result also holds for the Garsia's $BMOA$ and Bloch seminorms.

\begin{proposition}\label{prop:Bn:nested}
	Let $m,n\in\N$, $m<n$, and $f\in\H(\D)$. Then,
	\begin{align}
	\vertiii{f^m}_{BMOA}^{1/m}
	&\le \vertiii{f^n}_{BMOA}^{1/n}
	\label{BMOA:ineq}\\
	\vertiii{f^m}_{\B}^{1/m}
	&\le \vertiii{f^n}_{\B}^{1/n}.
	\label{Bloch:ineq}
	\end{align}
In particular, if $f^n\in BMOA$	($f^n\in\B$), then 
$f^m\in BMOA$	($f^m\in\B$). Moreover, if $f^n\in VMOA$	($f^n\in\B_0$), then 
$f^m\in VMOA$	($f^m\in\B_0$).
\end{proposition}

Bearing in mind that $f\mapsto f\circ\phi_a$ maps $H^2$ or $A^2$ to itself, Proposition~\ref{prop:Bn:nested} follows from the following lemma. 

\begin{lemma}
	Let $m,n\in\N$, $m<n$. Then:
	\begin{align}
	\|f^m-f^m(0)\|_{H^2}^{1/m}
	&\le \|f^n-f^n(0)\|_{H^2}^{1/n}
	\qquad(f\in\H(\D))
	\label{H2:ineq0}\\
	\|f^m-f^m(0)\|_{A^2}^{1/m}
	&\le \|f^n-f^n(0)\|_{A^2}^{1/n}
	\qquad(f\in\H(\D)).
	\label{A2:ineq0}
	\end{align}
\end{lemma}

\vspace*{1pt}
\begin{proof}
	We only prove \eqref{H2:ineq0}, the proof of  \eqref{A2:ineq0} is completely analogous replacing $H^2$ by $A^2$.
	First of all, recall that
	\begin{equation}\label{SGProp:eq1}
	\|f^k\|_{H^2}^2=\|f\|_{H^{2k}}^{2k}\qquad(f\in\H(\D),\,k\in\N),
	\end{equation}
	and, by Jensen's inequality, 
	\begin{equation}\label{SGProp:eq2}
	\|f\|_{H^{2n}}\ge\|f\|_{H^{2m}}\qquad(f\in\H(\D)).
	\end{equation}
	Now \eqref{H2:ineq0}, in the case $f(0)=0$, directly follows from~{\eqref{SGProp:eq1}} and~{\eqref{SGProp:eq2}}. Indeed, we have that
	\[
	\|f^n\|^2_{H^2}=\|f\|^{2n}_{H^{2n}}\ge\|f\|^{2n}_{H^{2m}}=\|f^m\|_{H^2}^{(2n)/m}
	\quad(f\in\H(\D)),
	\]
	and so
	\begin{equation}\label{SGProp:eq3}
	\|f^n\|^{1/n}_{H^2}\ge\|f^m\|^{1/m}_{H^2}\quad(f\in\H(\D)),
	\end{equation}
	which, in particular, gives \eqref{H2:ineq0} when $f(0)=0$. \vspace*{3pt}
	
	The general case is a consequence of~{\eqref{SGProp:eq2}}, \eqref{SGProp:eq3}, and a simple argument. First note that
	\begin{equation}\label{SGProp:eq4}
	\|f^k-f^k(0)\|^2_{H^2}=\|f^k\|^2_{H^2}-|f(0)|^{2k}
	\quad(f\in\H(\D),\,k\in\N).
	\end{equation}
	Then, for any $f\in\H(\D)$, we have that
	\begin{align*}
	\|f^n-f^n(0)\|^2_{H^2}
	&\stackrel{\mbox{\tiny$(\ast)$}}{=}
	\|f^n\|^2_{H^2}-|f(0)|^{2n}\\
	&\stackrel{\mbox{\tiny$(\star)$}}{\ge}
	\|f^m\|^{(2n)/m}_{H^2}-|f(0)|^{2n}\\
	&\stackrel{\mbox{\tiny$(\ast)$}}{=}
	\left(\|f^m-f^m(0)\|^2_{H^2}
	+|f(0)|^{2m}\right)^{n/m}-|f(0)|^{2n}\\
	&\stackrel{\mbox{\tiny$(\diamond)$}}{\ge}
	\|f^m-f^m(0)\|^{(2n)/m}_{H^2},
	\end{align*}
	where $(\ast)$ and $(\star)$ follow from~{\eqref{SGProp:eq4}} and~{\eqref{SGProp:eq3}}, respectively, while $(\diamond)$ is a consequence of the classical superadditivity inequality
	\[
	(x+y)^{\alpha}\ge x^{\alpha}+y^{\alpha}\qquad(x,y\ge0,\,\alpha\ge1).
	\] 
	(Recall that any convex function $\varphi:[0,\infty)\to\R$
	whith $\varphi(0)=0$ is superadditive, {\em i.e.} $\varphi(x+y)\ge\varphi(x)+\varphi(y)$, for any 
	$x,y\ge0$.)\newline
	Hence 
	\[
	\|f^n-f^n(0)\|^{1/n}_{H^2}\ge \|f^m-f^m(0)\|^{1/m}_{H^2}
	\qquad(f\in\H(\D)),
	\]
	and the proof is complete.
\end{proof}

The final part of this section recalls the descriptions of the symbols $g\in\H(\D)$ such that 
the operators $T_g$, $S_g$ and $M_g$  are bounded, or compact, on $A^p_{\alpha}$.

\begin{theorem}
\label{thm:TSM-bounded}
Let $g\in\H(\D)$, $0<p<\infty$ and $\alpha\ge-1$. Then:
\begin{enumerate}[label={\sf\alph*)},ref={\sf\alph*)},topsep=3pt, leftmargin=*,itemsep=3pt] 
\item \label{item:thm:TSM-bounded:1}
$T_g\in\BB(A^p_\alpha)$ if and only if $g\in \B$, when $\alpha>-1$, and $g\in BMOA$, when $\alpha=-1$.
 Moreover, $\|T_g\|_{\alpha,p}\simeq\|g\|_{\B}$, if $\alpha>-1$, and $\|T_g\|_{\alpha,p}\simeq\|g\|_{BMOA}$, if $\alpha=-1$.
\item \label{item:thm:TSM-bounded:2}
$T_g\in\KK(A^p_{\alpha})$ if and only if $g\in\B_0$, when $\alpha>-1$, and $g\in VMOA$, when $\alpha=-1$.
\end{enumerate}
\end{theorem}

Theorem \ref{thm:TSM-bounded} is originally proved, for $\alpha=-1$, in  \cite[Thm.~1, Corollary 1]{Aleman:Siskakis:1} ($p\ge1$) and in \cite[Thm.~1(ii), Corollary 1(ii)]{Aleman:Cima} ($0<p<1$)  and, for $\alpha>-1$,  in
\cite[Thm.~1]{Aleman:Siskakis:2} ($p\ge1$) and in 
\cite[Thm.~4.1(i)]{Aleman:Constantin} ($0<p<1$).

\begin{proposition}\label{prop:SM-bounded-compact}
Let $g\in\H(\D)$, $0<p<\infty$ and $\alpha\ge-1$. Then:
\begin{enumerate}[label={\sf\alph*)},ref={\sf\alph*)},topsep=3pt, leftmargin=*,itemsep=3pt] 
\item \label{item:prop:SM-bounded-compact:1}
$S_g\in\BB(A^p_\alpha)$ (or $M_g\in\BB(A^p_\alpha)$)  if and only if $g\in H^\infty$. 
Moreover, $\|S_g\|_{\alpha,p}\simeq\|M_g\|_{\alpha,p}\simeq\|g\|_{H^{\infty}}$.
\item \label{item:prop:SM-bounded-compact:2}
$S_g\in\KK(A^p_{\alpha})$ (or $M_g\in\KK(A^p_{\alpha})$) if and only if $g\equiv0$.
\end{enumerate}
\end{proposition}

The characterization of the boundedness for $M_g$ follows from  a classical result on pointwise multipliers (see~{\cite[Lemma 11]{Duren:Romberg:Shields}} or \cite[Lemma 1.10]{Vinogradov}). 
The remaining part of 
 Proposition \ref{prop:SM-bounded-compact} is well known for the experts, but unfortunately we have not found any explicit reference. For a sake of completeness we include a sketch of the proof. If $g\in H^{\infty}$ then $M_g,T_g,g(0)\delta_0\,\in\BB(A^p_ {\alpha})$, and so 
 $S_g\in\BB(A^p_{\alpha})$, by \eqref{eqn:formula}.
 In order to prove the converse, recall that the Bergman kernel for $A^2_{\alpha}$ is $K_{\alpha}(z,\lambda)=(1-\overline{\lambda}z)^{-\alpha-2}$, and, in particular, the analytic function 
 \[
 h_{\lambda}(z)=\frac{(1-|\lambda|^2)^{\frac{\alpha+2}p}}{(1-\overline{\lambda}z)^{\frac{2\alpha+4}p}}
\qquad(\lambda\in\D)
 \]
 satisfies $\|h_ {\lambda}\|_ {\alpha,p}=1$. Thus if $S_g\in\BB(A^p_{\alpha})$ then 
 \[
 |(S_gh_{\lambda})'(\lambda)|
 \le \frac{c_{\alpha,p}}{(1-|\lambda|^2)^{\frac{\alpha+2}p+1}}
 \|S_gh_ {\lambda}\|_ {\alpha,p}
 \le \frac{c_{\alpha,p}\|S_g\|_ {\alpha,p}}{(1-|\lambda|^2)^{\frac{\alpha+2}p+1}},
 \]
from which follows that 
$\|g\|_{H^{\infty}}\le C_{\alpha,p}\,\|S_g\|_ {\alpha,p}$.
A similar argument shows that if  $M_g\in\BB(A^p_{\alpha})$ then $g\in H^{\infty}$.

Using standard arguments on compact operators between spaces of analytic functions (see Lemma \ref{lem:compactness}) together with the above estimates it is easy to prove part \ref{item:prop:SM-bounded-compact:2} of Proposition \ref{prop:SM-bounded-compact}.

\section{The algebra $\A_g$ generated by the operators $T_g$, $S_g$, and $M_g$.}\label{sec:algebra:of:g-operators}

The main goal of this section is to show that any operator $L\in\A_g$  has a unique representation of the form  \eqref{eqn:general:expression:operator} when $g$ is non constant and $g(0)\neq 0$. A powerful purely algebraic machinery which helps dealing with such questions are the Gr\"obner bases 
\cite{AdLo}, \cite{Ufn}. However, we have preferred a direct approach, partly for the sake of completeness, but also  because our further arguments need some more specific information about this representation, like for example Proposition~\ref{prop:algebraLgk} below.

\subsection{Some useful identities} In this section we gather some formulas that will be used later on.

\begin{proposition}
Let $g\in \H(\D)$, and $j,k\in\N$. Then:
\begin{gather}
M_g   = S_g+T_g+g(0)\,\delta_0 \label{eqn:MSTdelta}\\
M_g^k = M_{g^k} \label{eqn:Mk}\\
S_g^k = S_{g^k} \label{eqn:Sk}\\
S_{g^j}T_{g^k}=\tfrac{k}{j+k}\,T_{g^{j+k}} \label{eqn:ST}\\
S_{g^j}M_{g^k}=S_{g^{j+k}}+\tfrac{k}{j+k}\,T_{g^{j+k}} 
                                           \label{eqn:SM}\\
T_{g^j}M_{g^k} = \tfrac{j}{j+k}\,T_{g^{j+k}} \label{eqn:TM}\\
T_gS_g = S_gT_g-T_g^2-g(0)(g-g(0))\,\delta_0                                   \label{eqn:TS}
\end{gather}
\end{proposition}
\vspace*{1pt}
\begin{proof} Let $f\in\H(\D)$.\hfill
\begin{enumerate}[itemsep=3pt,leftmargin=*,wide=0pt]
\item[\eqref{eqn:MSTdelta}]  Since $(gf)'=g'f+gf'$, we have
\[
g(z)f(z)
= g(0)f(0)+\int_0^zg'(\zeta)f(\zeta)\,d\zeta
   +\int_0^zg(\zeta)f'(\zeta)\,d\zeta,
\]
that is, $M_gf=T_gf+S_gf+g(0)\,\delta_0f$.
\item[\eqref{eqn:Mk}] $M_g^k f=g^k f=M_{g^k}f$.
\item[\eqref{eqn:Sk}] We proceed by induction on $k$. For $k=1$ there is nothing to prove. Now assume that $S_g^k = S_{g^k}$. Then
\[
S^{k+1}_gf(z)=S_g(S_{g^k}f)(z)
=\int_0^zg(\zeta)^{k+1}f'(\zeta)\,d\zeta
=S_{g^{k+1}}f(z),
\]
 that is, $S^{k+1}_g=S_{g^{k+1}}$.
\item[\eqref{eqn:ST}] It follows by integration from the identity   
$g^j(T_{g^k}f)'=\frac{k}{j+k}\,(g^{j+k})'f$.
\item[\eqref{eqn:SM}] It follows from~{\eqref{eqn:MSTdelta}},
\eqref{eqn:Sk} and \eqref{eqn:ST}:
\[
S_{g^j}M_{g^k}
=S_{g^j}S_{g^k}+S_{g^j}T_{g^k}
=S_{g^{j+k}}+\tfrac{k}{j+k}\,T_{g^{j+k}}
\]
\item[\eqref{eqn:TM}] It follows by integration from the identity   
$(g^j)'M_{g^k}f=\frac{j}{j+k}\,(g^{j+k})'f$.
\item[\eqref{eqn:TS}] It follows from~{\eqref{eqn:MSTdelta}},
\eqref{eqn:TM} and~{\eqref{eqn:ST}} :
\begin{align*}
T_gS_g 
&= T_gM_g-T_g^2-g(0)(g-g(0))\,\delta_0 \\ 
&=S_gT_g-T_g^2-g(0)(g-g(0))\,\delta_0\qedhere
\end{align*}
\end{enumerate}
\end{proof}

\begin{proposition} Let $g\in \H(\D)$, then
\begin{alignat}{1}
T_g(g-g(0))^n=\tfrac1{n+1}(g-g(0))^{n+1}\qquad(n\in\N\cup\{0\})
\label{eqn:Tg:on:powers:of:g-g(0)}\\
S_g(g-g(0))^n=g(0)(g-g(0))^n+\tfrac{n}{n+1}(g-g(0))^{n+1}
\qquad(n\in\N)
\label{eqn:Sg:on:powers:of:g-g(0)}
\end{alignat}
\end{proposition}

\begin{proof}
Identity~{\eqref{eqn:Tg:on:powers:of:g-g(0)}} is a direct computation, while~{\eqref{eqn:Sg:on:powers:of:g-g(0)}} is easily checked:
\begin{align*}
S_g(g-g(0))^n(z)
&=n\int_0^zg(\zeta)g'(\zeta)(g(\zeta)-g(0))^{n-1}d\zeta \\
&=n\int_0^zg'(\zeta)(g(\zeta)-g(0))^nd\zeta \\
& \,\,\,\,\,\,\,\,  
     +ng(0)\int_0^zg'(\zeta)(g(\zeta)-g(0))^{n-1}d\zeta \\
&= \frac{n}{n+1}(g(z)-g(0))^{n+1}+g(0)(g(z)-g(0))^n.
\qedhere
\end{align*}
\end{proof}

\begin{corollary}\label{cor:T:S:polynomial:g0}
Let $g\in \H(\D)$ and  let $P$ be a polynomial of degree $n$. Then:

\begin{enumerate}[label={\sf\alph*)},ref={\sf\alph*)},topsep=3pt, leftmargin=*,itemsep=3pt] 
\item \label{item:cor:T:S:polynomial:g0:1} 
 $T_g P(g-g(0))=Q(g-g(0))$, where $Q(z)=\int_0^zP(\zeta)\,d\zeta$ is a polynomial of degree $n+1$.
\item \label{item:cor:T:S:polynomial:g0:2}
 $S_g P(g-g(0))=Q(g-g(0))$, with 
$Q(z)=g(0)(P(z)-P(0))+\int_0^z\zeta\,P'(\zeta)\,d\zeta$, which is a polynomial of degree of $n+1$.
\end{enumerate}
\end{corollary}

\begin{proof}
Part~{\ref{item:cor:T:S:polynomial:g0:1}}  directly follows from~{\eqref{eqn:Tg:on:powers:of:g-g(0)}}. Part~{\ref{item:cor:T:S:polynomial:g0:2}}
is a direct consequence of~{\eqref{eqn:Sg:on:powers:of:g-g(0)}}
and the fact that $S_g1=0$.
\end{proof}

\begin{corollary}\label{cor:ST:mon}
Let $g\in \H(\D)$ and let $m,n\in\N$. Then 
\[
S^{m-j}_gT^j_g(g-g(0))^n
=\tfrac{n!}{(m+n)(n+j-1)!}\,(g-g(0))^{m+n}+P(g-g(0))
\quad(0\le j\le m),
\]
where $P$ is a polynomial of degree less than $m+n$ and whose coefficients only depend on $g(0)$, $m$, $n$ and $j$.
\end{corollary}

\begin{proof}
By~{\eqref{eqn:Tg:on:powers:of:g-g(0)}} it is clear that
\[
T_g^j(g-g(0))^n
=\tfrac1{(n+1)\cdots(n+j)}(g-g(0))^{n+j}
=\tfrac{n!}{(n+j)!}(g-g(0))^{n+j}.
\]
But~{\eqref{eqn:Sg:on:powers:of:g-g(0)}} gives that
\[
S^{m-j}_g(g-g(0))^{n+j}=\tfrac{n+j}{m+n}\,(g-g(0))^{m+n}+Q(g-g(0)),
\]
where $Q$ is a polynomial of degree less than $m+n$ whose coefficients only depend on $g(0)$, $m$, $n$ and $j$. Hence the proof is complete.
\end{proof}

\subsection{Vector space structure of $\A_g$}\label{sec:Vector space structure}

\begin{definition}\label{def:ST-decomposition}
Let $L\in\A_g^{(n)}$, where $n\in\N$.
We say that $L$ admits an {\em $ST$-decomposition} if
there exists a polynomial $P$ of degree less than $n$ satisfying 
\begin{equation*}
	L=\sum_{k=1}^n \sum_{j=0}^kc_{j,k}S_g^jT_g^{k-j}+ g(0)P(g-g(0))\,\delta_0,
\end{equation*} 
where $c_{j,k}\in\C$, for any $j,k$.  
\end{definition}

\begin{proposition}\label{prop:ST-decomposition} 
Let $g\in\H(\D)$ and $n\in\N$. Then
every  $L\in\A_g^{(n)}$ admits an $ST$-decomposition.
\end{proposition}

\begin{proof}
We proceed by induction on $n$.
For $n=1$ there is nothing to prove because~{\eqref{eqn:MSTdelta}} holds. 
  Let $n>1$. Since, by the induction hypothesis, any  $m$-letter $g$-word, with $m\le n-1$, admits an $ST$-decomposition, we will complete the proof by induction once we have checked that $L^{(n)}=L_nL^{(n-1)}$ has an $ST$-decomposition, when $L_n$ is either $S_g$, $T_g$, or $M_g$ and $L^{(n-1)}$ is either $g(0)P(g-g(0))\,\delta_0$, where $P$ is a polynomial of degree less than $n-1$, or $S_g^jT_g^{k-j}$, where $0\le j\le k$ and $1\le k\le n-1$. 

Assume first that $L^{(n-1)}=g(0)P(g-g(0))\,\delta_0$.
By the identity~{\eqref{eqn:MSTdelta}}  we only need to consider the case when $L_n$ is either $T_g$ or $S_g$. Then, by Corollary~{\ref{cor:T:S:polynomial:g0}}, $L^{(n)}=g(0)Q(g-g(0))\,\delta_0$, where $Q$ is a polynomial of degree less than $n$. 

Now assume that $L^{(n-1)}=S_g^jT_g^{k-j}$. As above, we only need to consider the cases $L_n=S_g$ and $L_n=T_g$. If $L_n=S_g$ then
$L^{(n)}=S_g^{j+1}T_g^{k-j}$, and, in particular, $L^{(n)}$ has an $ST$-decomposition. Now consider the case $L_n=T_g$.
If $j=0$ then $L^{(n)}=T_g^{k+1}$ and we are done.
If $j=k=1$, then 
$L^{(n)}=T_gS_g=S_gT_g-T_g^2-g(0)(g-g(0))\,\delta_0$, 
by~{\eqref{eqn:TS}}, so we also are done. Finally, if $j>1$ and $k>1$ then, again by~{\eqref{eqn:TS}}, we have that
\[
L^{(n)}=S_gT_gS_g^{j-1}T_g^{k-j}-T_g^2S_g^{j-1}T_g^{k-j},
\]
because $\delta_0S_g=0$. Since $T_gS_g^{j-1}T_g^{k-j}$ and $T_g^2S_g^{j-1}T_g^{k-j-1}$ are $g$-words with less than $n$ letters, they admit $ST$-decompositions, by the induction hypothesis. It directly follows that $L^{(n)}$ also has an $ST$-decomposition.
\end{proof}

From now on, in order to simplify the notation, we will write $g_0=g-g(0)$.
 By the above proposition, any non-trivial $g$-operator $L$ can be written as 
 \begin{equation}\label{eqn:general:formula:word:operator}
L=\sum_{k=0}^nS_g^kT_gP_k(T_g)+S_gP_{n+1}(S_g)+g(0)P_{n+2}(g_0)\,\delta_0,
\end{equation}
where $n\in\N\cup\{0\}$ and $P_0,\dots,P_{n+2}$ are polynomials  
such that $\deg P_{n+2}<n$ and either $P_n\ne0$ or $P_{n+1}\ne0$.
In other
words, the vector space $\A_g$ is spanned by
$\{S_g^jT_g^k:j,k\in\N\cup\{0\},\,j+k\ge1\}
\cup\{(g_0)^j\,\delta_0: j\in\N\cup\{0\}\}$
when $g(0)\ne0$, and by
$\{S_g^jT_g^k:j,k\in\N\cup\{0\},\,j+k\ge1\}$,
 when $g(0)=0$.

Our next goal is proving the uniqueness of the $ST$-decomposition
when the symbol $g$ is non constant and $g(0)\neq 0$.
We will need two preliminary results.

\begin{proposition} \label{prop:algebraLgk}
Let $g\in\H(\D)$, and let 
$L=L_1+g(0)P(g-g(0))\delta_0$, where 
\[
L_1=\sum_{k=1}^m \sum_{j=0}^kc_{j,k}S_g^jT_g^{k-j}
\]
is a $g$-operator of order $m\in\N$ and $P$ is a polynomial of degree less than $m$.
	Then there exists an increasing sequence $\{n_i\}_i$ in $\N$ such that $L[(g-g(0))^{n_i}]=P_i(g-g(0))$, where $P_i$ is a polynomial of degree $m+n_i$.	
\end{proposition}
 
\begin{proof}
By Corollary~{\ref{cor:ST:mon}}, for $n\in\N$ and $0\le k\le m$, we have
\[
L [(g_0)^{n+k+1}]=L_1 [(g_0)^{n+k+1}]=\frac{(n+k+1)!}{m+n+k+1}
\,a_{k,n}\,(g_0)^{m+n+k+1}+P_k(g_0),
\]
where $P_k$ is a polynomial of degree less than $m+n+k+1$ and 
\[
a_{k,n}=\sum_{j=0}^{m} \frac{c_{j,m}}{(n+m-j+k)!}.
\]
Since $L_1$ has order $m$, $(c_{0,m},c_{1,m},\cdots,c_{m,m})\ne(0,0,\dots,0)$, so we have that
$(a_{0,n},a_{1,n},\dots,a_{m,n})\ne(0,0,\dots,0)$, provided that
\begin{equation}\label{eqn:determinant}
D^{(m)}_{n}:=\det\left(\frac1{(n+m-j+k)!}\right)_{j,k=0}^m\ne0, 
\end{equation}
and, in particular, there is some $0\le k\le m$ such that $Lg_0^{n+k+1}=P(g_0)$, where $P$ is a polynomial of degree $m+n+k+1$. Thus we only have to check~{\eqref{eqn:determinant}}.
In order to do that we recall the so called Pochhammer symbols:
\[
(k)_0=1\qquad (k)_{\ell}=k(k+1)\cdots (k+\ell-1)\qquad(k,\ell\in\N).
\]
Since 
\[
\frac1{(n+m-j+k)!}=\frac1{(n+m+k)!}\,(n+m-j+k+1)_{j},
\]
we have that $D^{(m)}_{n}=b_{n,m}\,\Delta^{(m)}_{n}$, where $b_{n,m}>0$ and
\[
\Delta^{(m)}_{n}:=
	\begin{vmatrix}
	(n+m+1)_0& (n+m+2)_0&\cdots &(n+2m+1)_0\\
	(n+m)_1 & (n+m+1)_1 & \cdots & (n+2m)_1 \\
	\vdots  & \vdots  & \ddots & \vdots  \\
	(n+1)_m & (n+2)_m & \cdots & (n+m+1)_m
	\end{vmatrix}.
\]
But $(\ell)_0=1$ and  $(\ell+1)_{j+1}-(\ell)_{j+1}=(j+1)(\ell+1)_j$, we have
\[
\Delta^{(m)}_{n}=
\begin{vmatrix}
1(n+m+1)_0          & \cdots & 1(n+2m)_0 \\
2(n+m)_1         & \cdots & 2(n+2m-1)_1 \\
\vdots           & \ddots & \vdots \\
m(n+2)_{m-1} & \cdots & m(n+m)_{m-1} 
\end{vmatrix},
\]
and so $\Delta^{(m)}_n=m!\,\Delta^{(m-1)}_{n+1}$. Since $\Delta^{(1)}_{n+m}=1$,
 we get~{\eqref{eqn:determinant}}.
\end{proof}

\begin{lemma}\label{lem:basis:for:Agn}
 Let $g\in\H(\D)$.
If $g$ is not constant then $\{g^n:n\in\N\cup\{0\}\}$ and $\{(g-g(0))^n:n\in\N\cup\{0\}\}$ are bases for the vector space $\{P(g): P\mbox{ polynomial }\}$.
\end{lemma}

\begin{proof}
It is clear that $\{g^n:n\in\N\cup\{0\}\}$ and $\{(g_0)^n:n\in\N\cup\{0\}\}$ span the vector space 
$\{P(g): P\mbox{ polynomial }\}$. Now we want to prove that $\{g^n:n\in\N\cup\{0\}\}$ is linearly independent, which means that if $P(g)=0$, for some polynomial $P$, then $P\equiv0$. Thus assume that $P(g)=0$, for some polynomial $P$. Since $g$ is not constant, $g$ takes infinitely many values. It follows that $P$ has infinitely many zeros, that is, $P\equiv0$.
 A similar argument shows that $\{(g_0)^n:n\in\N\cup\{0\}\}$ is linearly independent, so the proof is complete.
\end{proof}

\begin{proposition}\label{prop:basis:for:Agn}
 Let $g\in\H(\D)$.

\begin{enumerate}[label={\sf\alph*)},topsep=3pt, leftmargin=*,itemsep=3pt] 
\item If $g\not\equiv0$ is constant and $I$ is the identity mapping on $\H(\D)$, then $\{I,\,\delta_0\}$ is a basis for $\A^{(n)}_g$, for every $n\in\N$, and so it is also a basis for $\A_g$.
\item If $g$ is not constant and $g(0)=0$, then \begin{equation}\label{eqn:basis:g(0)=0}
\{S_g^jT_g^{k-j}:1\le k\le n,\,0\le j\le k\}
\end{equation} 
is a basis for $\A^{(n)}_g$, and so $\{S_g^jT_g^k:j,k\in\N\cup\{0\},\,j+k\ge1\}$  is a basis for $\A_g$.
\item If $g$ is not constant and $g(0)\ne0$, then
\begin{equation}\label{eqn:basis:g(0)ne0}
\{S_g^jT_g^{k-j}:1\le k\le n,\,0\le j\le k\}
\cup\{(g-g(0))^j\delta_0: 0\le j< n\} 
\end{equation}
is a basis for $\A^{(n)}_g$, and so 
\[
\{S_g^jT_g^k:j,k\in\N\cup\{0\},\,j+k\ge1\}
\cup\{(g-g(0))^j\delta_0: j\in\N\cup\{0\}\}
\]
 is a basis for $\A_g$.
\end{enumerate}
\end{proposition}


\begin{proof}

\begin{enumerate}[label={\sffamily{\alph*)}},topsep=3pt, 
leftmargin=0pt, itemsep=3pt, wide, listparindent=0pt, itemindent=6pt] 
\item Assume $g\equiv c\ne0$. Then $T_g=0$, $S_g=cI-c\delta_0$ and $M_g=cI$, so both $\A^{(n)}_g$ and $\A_g$ are spanned (as vector spaces) by $I$ and $\delta_0$. On the other hand, $I$ and $\delta_0$ are linearly independent. Indeed, if $\alpha I+\beta\delta_0=0$, for some $\alpha,\beta\in\C$, then
$\alpha f=(\alpha I+\beta\delta_0)f=0$, for $f(z)=z$, so $\alpha=0$, and therefore 
$\beta=(\alpha I+\beta\delta_0)1=0$.
\item Assume $g$ is not constant and $g(0)=0$. Then Proposition~{\ref{prop:ST-decomposition}} shows that  $\A^{(n)}_g$ is spanned by~{\eqref{eqn:basis:g(0)=0}}. On the other hand, the linear independence of~{\eqref{eqn:basis:g(0)=0}} follows from
Proposition~{\ref{prop:algebraLgk}}. Indeed, if
\begin{equation}\label{eqn:independence:b}
\sum_{k=1}^n\sum_{j=0}^kc_{j,k}S_g^{k-j}T_g^j=0,
\end{equation}
where $c_{j,k}\in\C$, then $c_{j,k}=0$, for any $1\le k\le n$ and $0\le j\le k$, since otherwise Proposition~{\ref{prop:algebraLgk}} shows that there is some $\ell\in\N$ such that 
\[
\left(\sum_{k=1}^n\sum_{j=0}^kc_{j,k}S_g^{k-j}T_g^j\right)g^{\ell}=P(g),
\]  
where $P$ is a non-constant polynomial, which is absurd, taking into account~{\eqref{eqn:independence:b}} and Lemma~\ref{lem:basis:for:Agn}.

\item Assume $g$ is not constant and $g(0)\ne0$. First, note that~{\eqref{eqn:MSTdelta}} shows that $\delta_0=\frac1{g(0)}(M_g-S_g-T_g)\in\A^{(1)}_g$, and so~{\eqref{eqn:Tg:on:powers:of:g-g(0)}} gives that
\[
(g_0)^j\,\delta_0=j!\,(T_g^j1)\,\delta_0
=j!\,T_g^j\delta_0\in\A^{(n)}_g\quad(0\le j<n).
\]
On the other hand, since Proposition~{\ref{prop:ST-decomposition}} shows that  $\A^{(n)}_g$ is spanned by~{\eqref{eqn:basis:g(0)ne0}}, we only have to prove the linear independence of~{\eqref{eqn:basis:g(0)ne0}}.
 Assume that 
\begin{equation}\label{eqn:independence:c}
\sum_{k=1}^n\sum_{j=0}^kc_{j,k}S_g^{k-j}T_g^j
+P(g_0)\,\delta_0=0,
\end{equation}
where $c_{j,k}\in\C$ and $P$ is a polynomial. Then $c_{j,k}=0$, for any $1\le k\le n$ and $0\le j\le k$, since otherwise Proposition~{\ref{prop:algebraLgk}} shows that there is some $\ell\in\N$ such that 
\[
\left(\sum_{k=1}^n\sum_{j=0}^kc_{j,k}S_g^{k-j}T_g^j+P(g_0)\delta_0\right)
(g_0)^{\ell}=Q(g_0),
\]  
where $Q$ is a non-constant polynomial,  which is absurd, taking into account~{\eqref{eqn:independence:c}}  and Lemma~\ref{lem:basis:for:Agn}.
Therefore 
\[
P(g_0)=P(g_0)\,\delta_0 1=0,
\] 
and a second application of Lemma~{\ref{lem:basis:for:Agn}} gives that $P\equiv0$.\qedhere
\end{enumerate}	
\end{proof}

We end this section by giving a second application of Propositions~{\ref{prop:ST-decomposition}} 
and~{\ref{prop:algebraLgk}} (and Lemma~{\ref{lem:basis:for:Agn}}) which clarifies the concept of trivial $g$-operator. We recall that $L\in\A_g$ is trivial if $L=g(0)P(g_0)\delta_0$, for some polynomial~{$P$.}

\begin{proposition}\label{prop:char:trivial:g-operators}
Let $g\in\H(\D)$.
\begin{enumerate}[label={\sf\alph*)},topsep=3pt, leftmargin=*,itemsep=3pt] 
\item\label{item:char:trivial:g-operators:a} 
If $g(0)=0$ and $L=P(g)\delta_0\in\A_g$, for some polynomial $P$, then $L=0$. 
\item\label{item:char:trivial:g-operators:b}  
A $g$-operator $L$ is trivial if and only if  $L(z^{\ell})=0$, for every $\ell\in\N$.
\end{enumerate}
\end{proposition}

\begin{proof}
Assume that $g(0)=0$ and $L=P(g)\delta_0\in\A_g$, for some polynomial $P$. 
If $g$ is constant then $g\equiv0$, so $M_g=S_g=T_g=0$, 
and therefore $\A_g=0$, which gives that $L=0$. 
When $g$ is not constant we proceed by contradiction. 
Suppose that $L\ne0$. Then $L\in\A_g^{(m)}$, for some 
$m\in\N$, so Propositions~\ref{prop:ST-decomposition} 
and~\ref{prop:algebraLgk} show that there is $n\in\N$ 
such that $Lg^n=Q(g)$, where $Q$ is a polynomial of 
degree $m+n$. But, since $g(0)=0$, $Lg^n=0$, so 
$Q(g)=0$, and Lemma~\ref{lem:basis:for:Agn} implies 
that $Q\equiv0$, which is a contradiction and finishes 
the proof of part~\ref{item:char:trivial:g-operators:a}.

Finally, we prove part~\ref{item:char:trivial:g-operators:b}.
Now assume that $L\in\A_g$. If $L$ is trivial, it is 
clear that $L(z^{\ell})=0$, for every $\ell\in\N$.
On the other hand, if $L(z^{\ell})=0$, for any $\ell\in\N$, then $LP=L(P(0))=P(0)(L1)$, for any polynomial~{$P$.} Now the continuity of $L:\H(\D)\to\H(\D)$ implies that $Lf=f(0)(L1)$, for any $f\in\H(\D)$, that is, $L=(L1)\,\delta_0$. But $L1=P(g_0)$, where $P$ is a polynomial, and, by part~\ref{item:char:trivial:g-operators:a}, we conclude that $L$ is trivial.
\end{proof}

\section{Main results}
\label{sec:boundedness:g-operators}

We start this section by studying 
the behaviour of the iterates of $T_g$. 

\begin{proposition}\label{prop:T_g-powers} 
Let $g\in\H(\D)$. If $n\in\N$, $n>1$, and $T_g^n\in\BB(A^p_{\alpha})$, then $T_g\in\BB(A^p_{\alpha})$ and there exists a constant $c_n>0$, which only depends on $n$, such that
\begin{equation}\label{eqn:T_g-powers:1}
\|T_gf\|^n_{\alpha,p}\le c_n\,\|T^n_gf\|_{\alpha,p}\|f\|_{\alpha,p}^{n-1}
\qquad f\in A^p_\alpha,
\end{equation}
and so
\begin{equation}\label{eqn:T_g-powers:2}
\|T_g\|_{\alpha,p}^n\le c_n\,\|T_g^n\|_{\alpha,p}.
\end{equation}
In particular, $T^n_g\in\BB(A^p_{\alpha})$, for some $n\in\N$,  if and only $T^n_g\in\BB(A^p_{\alpha})$, for any $n\in\N$.

\end{proposition}
In order to prove Proposition~{\ref{prop:T_g-powers}} we need the following useful result, which is proved in 
\cite[Thm.\ 1 (i)]{Aleman:Cima} for $\alpha=-1$, while for $\alpha>-1$ it is a direct consequence of H\"older's inequality and the fact that the differentiation operator $f\mapsto f'$ is a topological isomorphism from
$A^p_{\alpha}(0)=\{f\in A^p_ {\alpha}:f(0)=0\}$ onto $A^p_{\alpha+p}$ 
\cite[Thm.\ 4.28]{Zhu}.

\begin{lemma}\label{lem:p-to-q} Let $r,q,s>0,~\frac1{r}+\frac1{s}=\frac1{q}$, and $g\in A_{\alpha}^r$.  Then $T_g:A_{\alpha}^s\to A_{\alpha}^q$ is bounded and there exists a constant $c>0$, independent of $g$,  satisfying that $\|T_g\|_{A_{\alpha}^s\to A_{\alpha}^q}\le c\,\|g\|_{\alpha,r}$.  
\end{lemma}

\begin{proof}[\bf Proof of Proposition~{\ref{prop:T_g-powers}}]  
Note that $T_g^n1$ is a polynomial of degree $n$ in $g$, so that $g^k\in A_\al^p,~1\le k\le n$.  Inductively it follows easily that   $g^k\in A_\al^p$, for all $k\ge 1$. Then using integration by parts we see that  $T_g^kf\in A_\al^p$ whenever $k\ge 1$ and $f$ is a polynomial.
For $k>1$, we apply Lemma~{\ref{lem:p-to-q}}  with $r=s=p,~q=p/2$, to conclude that if $f$ is a polynomial and $h\in A_\al^p$ then $T_{T_g^kf}h\in A_\al^{\frac{p}{2}}$ with   
\[
\|T_{T_g^kf}h\|_{\al,\frac{p}{2}}\lesssim \|T_g^kf\|_{\al,p}\|h\|_{\al,p}.
\]
Now, for $k\ge2$, let $h=T^{k-2}_gf$ and note that 
\[
T_{T_g^kf}h(z)=\int_0^zg'(\zeta)(T_g^{k-2}f)(\zeta)
(T_g^{k-1}f)(\zeta)\,d\zeta=\tfrac1{2}(T_g^{k-1}f)^2(z).
\]
Since $\|(T_g^{k-1}f)^2\|_{\alpha,\frac{p}2}=\|T_g^{k-1}f\|^2_{\al, p}$, this leads to the estimate 
\begin{equation}\label{eqn:T_g-powers:estimate1}
\|T_g^{k-1}f\|^2_{\al, p}\lesssim\|T_g^kf\|_{\al,p}\|T_g^{k-2}f\|_{\al,p}
\qquad(k\ge2).
\end{equation}
	
By induction on $j\ge 1$, from~{\eqref{eqn:T_g-powers:estimate1}} we obtain \begin{equation}\label{eqn:T_g-powers:estimatej}
\|T_g^{k-j}f\|^{j+1}_{\al, p}\lesssim\|T_g^kf\|_{\al,p}\|T_g^{k-j-1}f\|^j_{\al,p}
\qquad(k\ge j+1).
\end{equation}
Indeed, assume that 
\begin{equation}\label{eqn:induc-j-1}
\|T_g^{k-j+1}f\|^{j}_{\al, p}\lesssim\|T_g^kf\|_{\al,p}\|T_g^{k-j}f\|^{j-1}_{\al,p}
\qquad(k\ge j),
\end{equation}
and we want to obtain \eqref{eqn:T_g-powers:estimatej}.
	
By \eqref{eqn:T_g-powers:estimate1}, for each $k\ge j+1$ we have
\[
\|T_g^{k-j} f\|_{\alpha,p}^{2j}
\lesssim \|T_g^{k-j+1} f\|_{\alpha,p}^{j}
\|T_g^{k-j-1} f\|_{\alpha,p}^j.
\]
Now, by \eqref{eqn:induc-j-1}, we obtain
\[
\|T_g^{k-j} f\|_{\alpha,p}^{2j}
\lesssim \|T_g^{k} f\|_{\alpha,p}
\|T_g^{k-j} f\|_{\alpha,p}^{j-1}\|T_g^{k-j-1} f\|_{\alpha,p}^j,
\]
which proves \eqref{eqn:T_g-powers:estimatej}. 

Finally, the estimates~{\eqref{eqn:T_g-powers:estimate1}}
and~{\eqref{eqn:T_g-powers:estimatej}}  for $k=2$ and $k-j=2$, respectively, give that
\begin{align*}
\|T_gf\|^2_{\alpha,p}
&\lesssim\|T_g^2f\|_{\alpha,p}\|f\|_{\alpha,p}\\ 
\|T_g^2f\|_{\al,p}^{k-1}
&\lesssim\|T_g^kf\|_{\al,p}\|T_gf\|_{\al,p}^{k-2}
\qquad(k\ge 3).
\end{align*}
Therefore  
\[
\|T_gf\|_{\alpha,p}^{2(k-1)}
\lesssim \|T_g^2f\|_{\alpha,p}^{k-1}\|f\|_{\alpha,p}^{k-1}
\lesssim \|T_g^kf\|_{\alpha,p} \|T_gf\|_{\alpha,p}^{k-2}\|f\|_{\alpha,p}^{k-1}
\qquad(k\ge3),
\]
and so
\[
\|T_gf\|_{\alpha,p}^k
\lesssim \|T_g^kf\|_{\alpha,p}\|f\|_{\alpha,p}^{k-1},\quad \mbox{for any polynomial $f$}
\qquad(k\ge2).
\]
In particular, if $k=n$, bearing in mind that the polynomials are dense in $A^p_\alpha$, 
the preceding estimate shows that~{\eqref{eqn:T_g-powers:1}} holds, and, as a consequence, \eqref{eqn:T_g-powers:2} also holds. Hence $T_g\in\BB(A^p_{\alpha})$.
\end{proof}

For $h\in\H(\D)$ and  $\lambda\in\overline{\D}$, let us consider the dilated functions
\[
h_{\lambda}(z):=h(\lambda z), \qquad z\in\D.
\] 
The map $h\mapsto h_\lambda$  is a linear contractive  operator on $A^p_{\alpha}$. Moreover, 
\begin{equation}\label{eq:dilation:MST}
(M_gf)_\lambda=M_{g_\lambda}f_\lambda\qquad
(S_gf)_\lambda=S_{g_\lambda}f_\lambda\qquad
(T_gf)_\lambda=T_{g_\lambda}f_\lambda.
\end{equation}
Now a repeated application of~{\eqref{eq:dilation:MST}} shows that
\begin{equation}\label{eq:dilation:operator:function}
L_{g_{\lambda}}f_{\lambda}=(L_gf)_{\lambda}\qquad(L_g\in\A_g).
\end{equation}

The following result is a key tool in our study of the boundedness of operators in $\A_g$.

\begin{proposition}\label{prop:norm:dilations}  
Let $g\in\H(\D)$ and let $L_g\in\A_g$.
If  $L_g\in\BB(A_\al^p)$  then $L_{g_{\lambda}}\in\BB(A_\al^p)$ and $\|L_{g_{\lambda}}\|_{\alpha,p}\le \|L_g\|_{\alpha,p}$, for any $\lambda\in\overline{\D}$.
Moreover, if ${\displaystyle\varliminf_{r\nearrow1}\|L_{g_r}\|_{\alpha,p}<\infty}$,
then $L_g\in\BB(A_{\alpha}^p)$ and 
${\displaystyle\|L_g\|_{\alpha,p}=\varliminf_{r\nearrow1}\|L_{g_r}\|_{\alpha,p}}$.
 \end{proposition}
\begin{proof}

First note that, for any $\lambda\in\T$,  {\eqref{eq:dilation:operator:function}}~gives that
$L_{g_{\lambda}}f=(L_gf_{\overline{\lambda}})_{\lambda}$
and, since $f\mapsto f_{\lambda}$ is an invertible isometry on $A_\al^p$, it follows  that  
$L_{g_{\lambda}}\in\BB(A_\al^p)$ and $\|L_{g_{\lambda}}\|_{\alpha,p}=\|L_g\|_{\alpha,p}$.
If $\lambda\in \D$, then $g_{\lambda}\in\H(\overline{\D})$, so $M_{g_{\lambda}},S_{g_{\lambda}},T_{g_{\lambda}}\in\BB(A^p_{\alpha})$, and, as a consequence, $L_{g_{\lambda}}\in\BB(A^p_{\alpha})$.

   In order to estimate the operator norm of $L_{g_{\lambda}}$, let $f$ be a polynomial  and  observe that, for fixed $z\in \D$, the function $\lambda\mapsto L_{g_{\lambda}}f(z)$ is analytic on $\D$. Indeed, this is an inmediate consequence of the fact that if $(\lambda,z)\mapsto h(\lambda,z)$ is an analytic function on the bidisc $\D^2$ then 
\[
M_{g_{\lambda}}h(\lambda,\cdot)(z),\quad S_{g_{\lambda}}h(\lambda,\cdot)(z)\quad\mbox{and}\quad T_{g_{\lambda}}h(\lambda,\cdot)(z)
\] 
are also analytic functions of $(\lambda,z)$ on $\D^2$.

Next we are going to show that $F(\lambda)=L_{g_{\lambda}}f$ defines a continuous mapping from $\overline{\D}$ to 
$A^p_{\alpha}$. 

Assume first that $\zeta\in\D$.
For each $z\in\D$ the function 
$\lambda\mapsto L_{g_{\lambda}} f(z)$ is analytic on $\D$, which  implies  that $L_{g_{\lambda}} f(z)\to L_{g_{\zeta}}f(z)$, as $\lambda\to\zeta$. 
Since $L_{g_{\lambda}} f$ is uniformly bounded on $\D$, for $|\lambda-\zeta|<\frac12(1-|\zeta|)$,
 the Dominated Convergence Theorem shows that $\|F(\lambda)-F(\zeta)\|_{\al,p}\to0$, as $\lambda\to\zeta$.

If $\zeta\in\T$, we write,  by abuse of notation,  $f_{1/\lambda}(z)=f(z/\lambda)$, which  is well defined for  a polynomial $f$ and $\lambda\in\C\setminus\{0\}$. 
Then, by~{\eqref{eq:dilation:operator:function}}, for any $\lambda\in\overline{\D}\setminus\{0\}$ we have that
\[
F(\lambda)-F(\zeta)=L_{g_{\lambda}}f-L_{g_{\zeta}} f=
( L_gf_{1/\lambda}-L_gf_{\overline{\zeta}})_{\lambda}+
		(L_gf_{\overline{\zeta}})_{\lambda}-(L_gf_{\overline{\zeta}})_\z,
\] 
and so
\begin{align*}
\|F(\lambda)-F(\zeta)\|_{\al,p} 
&\le c \left[\|(L_gf_{1/\lambda}
                -L_gf_{\overline{\zeta}})_{\lambda}\|_{\al,p}
  +\|(L_gf_{\overline{\zeta}})_{\lambda}
               -(L_gf_{\overline{\zeta}})_\z\|_{\al,p}\right] \\
&\le c\left[\|L_gf_{1/\lambda}-L_gf_{\overline{\zeta}} \|_{\al,p}
  +\|(L_gf_{\overline{\zeta}})_{\lambda}
                -(L_gf_{\overline{\zeta}})_\z\|_{\al,p}\right] \\
&\le c\left[\|L_g\|\|f_{1/\lambda}
                -f_{\overline{\zeta}}\|_{\al,p}+
		\| (L_gf_{\overline{\zeta}})_{\lambda}-(L_gf_{\overline{\zeta}})_\z \|_{\al,p}\right],
\end{align*}
where $c=1$ if $p\ge 1$, and  $c=2^{1/p}$ if $0<p<1$.	
		Recall that $f$ is a polynomial and use the  elementary fact that, for $h\in A_{\al}^p$, $\|h_{\lambda}- h_\z\|_{\al,p}\to 0$, as $\lambda\to\zeta$,  to conclude that the right hand side converges to $0$ and therefore $\|F(\lambda)-F(\zeta)\|_{\al,p}\to0$, as $\lambda\to\zeta$. 

Hence we have just proved that $F:\overline{\D}\to A^p_{\alpha}$ is continuous, and, as a consequence, the function 
$u_f:\overline{\D}\to\C$, defined by
\[
u_f(\lambda)=\|F(\lambda)\|_{\al,p}^p=\|L_{g_{\lambda}}f\|_{\al,p}^p,
\]
 is also continuous. Moreover, since, for fixed $z\in\D$, $L_{g_{\lambda}}f(z)$ is an analytic function on $\lambda$, it is clear that $u_f$ is subharmonic in $\D$. It follows that $u_f$ attains its maximum at some point $\zeta\in\T$, which gives that
\[
\|L_{g_{\lambda}}f\|_{\alpha,p}=u_f(\lambda) \le \|L_{g_{\zeta}} f\|_{\alpha,p}\le \|L_{g_{\zeta}}\|\|f\|_{\al,p}=\|L_g\|\|f\|_{\alpha,p},
\]
for any $\lambda\in\overline{\D}$ and for any polynomial $f$.
Since the polynomials are dense in $A_\al^p$, 
we conclude that $\|L_{g_{\lambda}}\|_{\alpha,p}\le\|L_g\|_{\alpha,p}$.

Finally, for any $f\in A^p_{\alpha}$, Fatou's lemma shows that
\begin{align*}
\|L_gf\|_{\alpha,p}
&\le\varliminf_{r\nearrow1}\|(L_gf)_r\|_{\alpha,p}
=\varliminf_{r\nearrow1}\|L_{g_r}f_r\|_{\alpha,p} \\
&\le \varliminf_{r\nearrow1}\|L_{g_r}\|_{\alpha,p}\|f_r\|_{\alpha,p}
\le \varliminf_{r\nearrow1}\|L_{g_r}\|_{\alpha,p}\|f\|_{\alpha,p}.\qedhere
\end{align*}
\end{proof}

\subsection{Proof of Theorem~{\ref{thm:algebraimpliesginBloch:introduction}}}
 From now on we shall use repeatedly the following elementary fact:

\begin{remark}\label{rem:real-variable:boundedness}
 If a function $\varphi:[0,\infty)\to \mathbb{R}$ satisfies $\lim_{x\to\infty}\varphi(x)=\infty$, then the preimage by  $\varphi$ of any bounded set of real numbers is bounded.
\end{remark}

We will also need a couple of preliminary results.

\begin{lemma}\label{lem:P(g):bounded}
Let $g\in\H(\D)$ and let $P$ be a polynomial of degree $n\ge1$. If $P(g)\in H^{\infty}$, then $g\in H^{\infty}$.  
\end{lemma}

\begin{proof}
Assume that $P(g)\in H^{\infty}$, where $P(z)=\sum_{k=0}^na_kz^k$ is a polynomial of degree $n\ge1$. Then 
\[
|a_n||g(z)|^n-\sum_{k=0}^{n-1} |a_k||g(z)|^k\le\|P(g)\|_{\infty}\qquad(z\in\D),
\] 
and so Remark~{\ref{rem:real-variable:boundedness}}
completes the proof.
\end{proof}

\begin{lemma}\label{lem:polynomial:Tg:bounded}
Let $g\in\H(\D)$ and  let $P$ be a polynomial of degree $n\ge1$. If $P(T_g)\in\BB(A^p_{\alpha})$, then $T_g\in\BB(A^p_{\alpha})$.  
\end{lemma}

\begin{proof}
Assume that $P(T_g)\in\BB(A^p_{\alpha})$, where $P(z)=\sum_{k=0}^na_kz^k$ is a polynomial of degree $n\ge1$. Then, by Proposition~{\ref{prop:norm:dilations}},  
\[
|a_n| \|T_{g_r}^{n}\|_{\alpha,p}
-c_{n,p}\sum_{k=0}^{n-1} |a_k|\|T_{g_r}\|^k_{\alpha,p}
\le\|P(T_{g_r})\|_{\alpha,p}\le\|P(T_g)\|_{\alpha,p}
\quad(0<r<1).
\]
Now Proposition~{\ref{prop:T_g-powers}} shows that
$\|T_{g_r}^{n}\|_{\alpha,p}\ge c_n\|T_{g_r}\|_{\alpha,p}^{n}$, for some constant $c_{n}>0$ only dependent on~{$n$}. Thus $\varphi(\|T_{g_r}\|_{\alpha,p})\le\|P(T_g)\|_{\alpha,p}$, for every $0<r<1$, where 
$\varphi(x)=c_n|a_n|x^n-c_{n,p}\sum_{k=0}^{n-1}|a_k|x^k$.   
Hence Remark~{\ref{rem:real-variable:boundedness}} 
and Proposition ~{\ref{prop:norm:dilations}} end the proof.
\end{proof}

\begin{proof}[\bf Proof of Theorem \ref{thm:algebraimpliesginBloch:introduction}
\ref{item:thm:algebraimpliesginBloch:introduction:2}]
 Let be $P(z)=a_{m} z^{m}+Q(z)$, where $Q$ is a polynomial of degree less than $m$. Then, by~{\eqref{eq:dilation:operator:function}}  
\[
g(0)^2\, P((g_0)_r)=(L_gg)_r=L_{g_r}g_r
\qquad(0<r<1)
\]
so, since $\|L_{g_r}g_r\|_{\alpha,p}\le\|L\|_{\alpha,p}\|g_r\|_{\alpha,p}$ (see Proposition~{\ref{prop:norm:dilations}}), we obtain the estimate
\begin{align*}
|g(0)|^2|a_m|\|(g_0)_r^{m}\|_{\alpha,p}
&=\|L_{g_r}g_r-g(0)^2Q((g_0)_r)\|_{\alpha,p}\\
&\lesssim \sum_{j=0}^{m-1}\|(g_0)_r^j\|_{\alpha,p} 
\lesssim \sum_{j=0}^{m-1}\|(g_0)_r^{m}\|^{j/m}_{\alpha,p}.
\end{align*}
Therefore Remark~{\ref{rem:real-variable:boundedness}}
implies that $\sup_{0<r<1}\|(g_0)_r^{m}\|_{\alpha,p}<\infty$, and hence Fatou's lemma shows that $g_0^m\in A_{\alpha}^p$, which means that $g^m\in A^p_{\alpha}$. 
\end{proof}

Prior to proving  Theorem \ref{thm:algebraimpliesginBloch:introduction}
\ref{item:thm:algebraimpliesginBloch:introduction:1} some definitions and results about the theory of iterated commutators are needed.
Let $A,B:\H(\D)\to\H(\D)$ be two linear operators.
The {\it commutator} of $A$ and $B$ is the linear operator $[A,B]:=AB-BA$. If $C,D:\H(\D)\to\H(\D)$ are linear operators which commute with $B$ then
\begin{equation}\label{eqn:commutator:CAD:B:C:D:commute:B}
[CAD,B]=C[A,B]D.
\end{equation}
\noindent
The {\it iterated commutators} $[A,B]_k$, $k\in\N$, are defined inductively as follows:
\[
[A,B]_1:=[A,B]
\qquad\mbox{and}\qquad
[A,B]_{k+1}:=[[A,B]_k,B],\quad\mbox{for $k\in\N$.}
\]

We will use the following formula
\begin{equation}\label{eqn:formula:iterated:commutator}
[A,B]_k=\sum_{j=0}^k(-1)^j\binom{k}{j}B^jAB^{k-j}
\qquad(k\in\N).
\end{equation}

\begin{proposition} Let $g\in \H(\D)$ and $k\in\N$, then
\begin{equation}\label{eqn:commutator:Sgk:Tg}
[S^k_g,T_g]=T_gT_{g^k}+g(0)^k(g-g(0))\,\delta_0 
\end{equation}
\end{proposition}

\begin{proof}
By~{\eqref{eqn:ST}} and~{\eqref{eqn:TM}}  we have that
$S_g^kT_g=S_{g^k}T_g=T_gM_{g^k}$, so~{\eqref{eqn:MSTdelta}} gives that
\[
S_g^kT_g
= T_gT_{g^k}+T_gS_{g^k}+g(0)^kT_g\delta_0
= T_gT_{g^k}+T_gS_g^k+g(0)^kg_0\,\delta_0,
\]
which is just~{\eqref{eqn:commutator:Sgk:Tg}}.
\end{proof}

\begin{proposition} Let $g\in\H(\D)$ and $k\in\N$, then
\begin{equation}
\label{eqn:commutator:exp-power:g-g(0):delta0:Tg}
[\tfrac{(g-g(0))^k}{k!}\,\delta_0,T_g]_j
=(-1)^j\,\tfrac{(g-g(0))^{k+j}}{(k+j)!}\,\delta_0
\qquad(j,k\in\N).
\end{equation}
\end{proposition}

\begin{proof}
 Observe that~{\eqref{eqn:commutator:exp-power:g-g(0):delta0:Tg}} follows by induction on $j$ from~{\eqref{eqn:Tg:on:powers:of:g-g(0)}}. Indeed, \eqref{eqn:Tg:on:powers:of:g-g(0)} directly shows~{\eqref{eqn:commutator:exp-power:g-g(0):delta0:Tg}} for $j=1$:
\[
[\tfrac{(g_0)^k}{k!}\,\delta_0,T_g]
=-T_g\bigl(\tfrac{(g_0)^k}{k!}\bigr)\delta_0
=-\tfrac{(g_0)^{k+1}}{(k+1)!}\,\delta_0.
\]
Moreover, if $[\tfrac{(g_0)^k}{k!}\,\delta_0,T_g]_j
=(-1)^j\,\tfrac{(g_0)^{k+j}}{(k+j)!}\,\delta_0$ holds, 
then~{\eqref{eqn:Tg:on:powers:of:g-g(0)}}  gives that
\[
[\tfrac{(g_0)^k}{k!}\,\delta_0,T_g]_{j+1}
=(-1)^{j+1}
  T_g\bigl(\tfrac{(g_0)^{k+j}}{(k+j)!}\bigr)\delta_0
=(-1)^{j+1}\tfrac{(g_0)^{k+j+1}}{(k+j+1)!}\,\delta_0.
\qedhere
\]
\end{proof}

\begin{corollary}\label{co:iteratedcom} Let $g\in \H(\D)$ and $k\in\N$, then
\begin{equation}\label{eqn:iterated:commutators:Sgk:Tg}
[S_g^k,T_g]_j
=\tfrac{k!}{(k-j)!}\,T_g^jS_g^{k-j}T^j_g
 -\tfrac{(-1)^j}{j!}\,g(0)^k(g-g(0))^j\,\delta_0
\quad(1\le j\le k).
\end{equation}
Moreover, 
\begin{equation}\label{eqn:iterated:commutators:Sgk:Tg:0}
[S_g^k,T_g]_j=-\tfrac{(-1)^j}{j!}\,g(0)^k(g-g(0))^j\,\delta_0\qquad(j>k).
\end{equation}
\end{corollary}

\begin{proof}
We prove~{\eqref{eqn:iterated:commutators:Sgk:Tg}} by induction on $j$. 
For $j=1$, \eqref{eqn:iterated:commutators:Sgk:Tg} is just~{\eqref{eqn:commutator:Sgk:Tg}}. Now fix $1\le j<k$ and assume that
\[
[S_g^k,T_g]_j
=\tfrac{k!}{(k-j)!}\,T_g^jS_g^{k-j}T^j_g
 -\tfrac{(-1)^j}{j!}\,g(0)^k(g_0)^j\,\delta_0.
\] 
Then~{\eqref{eqn:commutator:CAD:B:C:D:commute:B}} and~{\eqref{eqn:Tg:on:powers:of:g-g(0)}} show that
\[
[S_g^k,T_g]_{j+1}
=\tfrac{k!}{(k-j)!}\,T_g^j[S_g^{k-j},T_g]T^j_g
  -\tfrac{(-1)^{j+1}}{(j+1)!}\,g(0)^k(g_0)^{j+1}\,\delta_0.
\]
Therefore~{\eqref{eqn:commutator:Sgk:Tg}} implies that
\[
[S_g^k,T_g]_{j+1}
=\tfrac{k!}{(k-j-1)!}\,T_g^{j+1}S_g^{k-j-1}T^{j+1}_g
  -\tfrac{(-1)^{j+1}}{(j+1)!}\,g(0)^k(g_0)^{j+1}\,\delta_0.
\]
Thus~{\eqref{eqn:iterated:commutators:Sgk:Tg}} is proved.
In particular, 
\[
[S_g^k,T_g]_k
=k!\,T_g^{2k}-\tfrac{(-1)^k}{k!}\,g(0)^k(g_0)^k\,\delta_0,
\]
and so, for $j>k$, \eqref{eqn:commutator:exp-power:g-g(0):delta0:Tg} implies that
\[
[S_g^k,T_g]_j
=-(-1)^kg(0)^k[\tfrac{(g_0)^k}{k!}\,\delta_0,T_g]_{j-k}
=-\tfrac{(-1)^j}{j!}g(0)^k(g_0)^j\,\delta_0.
\qedhere
\]
\end{proof}

\begin{proof}[\bf Proof of Theorem \ref{thm:algebraimpliesginBloch:introduction}
\ref{item:thm:algebraimpliesginBloch:introduction:1}]
First of all, we observe that if $L$ is not trivial and $L\in\BB(A^p_{\alpha})$, then $g^k\in A^p_\alpha$, for any $k\in\N$.
Indeed, Propositions~{\ref{prop:ST-decomposition}} and~{\ref{prop:algebraLgk}} show that there is an strictly increasing sequence $\{k_j\}$  in $\N$ such that $L(g_0^{k_j})=P_j(g_0)$, where $P_j$ is a polynomial of degree $d_j>k_j$. Then arguing as in the proof of part \ref{item:thm:algebraimpliesginBloch:introduction:2} we obtain that $g^{d_j}\in A^p_{\alpha}$, and consequently, $g^k\in A^p_{\alpha}$, for every $k\in\N$.

Now we prove part \ref{item:thm:algebraimpliesginBloch:introduction:1}.
Taking into account ~{Proposition~\ref{prop:ST-decomposition}}, Lemma~{\ref{lem:polynomial:Tg:bounded}}, and the above observation, we may assume that
\[
L_g=P_0(T_g)+\sum_{k=1}^nS_g^kP_k(T_g),
\] 
where $P_0,\dots,P_n$ are polynomials, and $P_n$ has degree $m\ge1$.

On the other hand, since $P_k(T_{g_r})$ commute  with $T_{g_r}$, \eqref{eqn:commutator:CAD:B:C:D:commute:B}, \eqref{eqn:iterated:commutators:Sgk:Tg} and~{\eqref{eqn:iterated:commutators:Sgk:Tg:0}}
give that
\begin{align*}
[L_{g_r},T_{g_r}]_n
&= n!\,T_{g_r}^{2n}P_n(T_{g_r})
                        +g(0)Q_0(g_r-g(0))\,\delta_0\\
&= Q_n(T_{g_r})+g(0)Q_0(g_r-g(0))\,\delta_0,
\end{align*}
where $Q_n$ and $Q_0$ are polynomials and $Q_n$ has degree $N=2n+m>n$.  
Now {\eqref{eqn:formula:iterated:commutator}} and  Proposition~{\ref{prop:norm:dilations}} imply that 
\[
\|[L_{g_r},T_{g_r}]_n\|_{\alpha,p}
\le c_{n,p}\|L_{g_r}\|_{\alpha,p}\|T_{g_r}\|_{\alpha,p}^n
\le c_{n,p}\|L_g\|_{\alpha,p}\|T_{g_r}\|_{\alpha,p}^n.
\]
Moreover, 
$\|Q_0(g_r-g(0))\,\delta_0\|_{\alpha,p}
\le \|Q_0(g_0)\,\delta_0\|_{\alpha,p}=C<\infty$, by Theorem~{\ref{thm:algebraimpliesginBloch:introduction}
\ref{item:thm:algebraimpliesginBloch:introduction:2}}
 and Proposition \ref{prop:norm:dilations}. 
On the other hand,  if $Q_n(z)=\sum_{k=0}^Na_kz^k$, then, taking into account Proposition~{\ref{prop:T_g-powers}}, we have 
\[
c_N|a_N|\|T_{g_r}\|^N_{\alpha,p}
-c_{N,p}'\sum_{k=0}^{N-1}|a_k|\|T_{g_r}\|_{\alpha,p}^k
\le \|Q_n(T_{g_r})\|_{\alpha,p}.
\]
Therefore, putting all that together, we get that $\varphi(\|T_{g_r}\|_{\alpha,p})\le C$, where
\[
\varphi(x)=c_N|a_N|x^N
-c_{N,p}'\sum_{k=0}^{N-1}|a_k| x^k
-c_{n,p}\|L_g\|_{\alpha,p}\,x^n.
\]
Hence Remark~{\ref{rem:real-variable:boundedness}} 
and  Proposition~{\ref{prop:norm:dilations}} conclude the proof.
\end{proof}

\subsection{Proof of Theorem~{\ref{thm:main:introduction}}}
In order to give the proof, we need the following  well known characterization of compact operators.

\begin{lemma}[{\cite[Lemma 3.7]{Tjani}}]\label{lem:compactness}
Let $X$ and $Y$ be two Banach (or quasi-Banach) spaces of
analytic functions on $\D$, and let $T:X\to Y$ be a linear operator. Suppose that the
following conditions are satisfied:
\begin{enumerate}[label={\sf(\alph*)},ref={\sf(\alph*)},topsep=6pt, leftmargin=32pt,itemsep=3pt] 
\item \label{item:lem:compactness:1}
The point evaluation functionals on $Y$ are bounded.
\item \label{item:lem:compactness:2}
The closed unit ball of $X$ is a compact subset of $\H(\D)$, where $\H(\D)$ is endowed with the topology of uniform convergence on compacta.
\item \label{item:lem:compactness:3} 
$T:X\to Y$ is continuous, where both $X$ and $Y$ are endowed with the topology of uniform convergence on compacta.
\end{enumerate}
Then $T:X\to Y$ is a compact operator if and only if for 
any bounded sequence $\{f_j\}$ in $X$ such that $f_j\to 0$ uniformly on compacta, the
sequence $\{Tf_j\}$ converges to zero in the norm of $Y$.
\end{lemma} 

 It is worth mentioning that conditions~{\ref{item:lem:compactness:1}} and~{\ref{item:lem:compactness:2}} of the previous lemma hold when $X=Y=A^p_\alpha$, and in such a case any $g$-operator satisfies~{\ref{item:lem:compactness:3}}.

\begin{proof}[\bf Proof of Theorem~{\ref{thm:main:introduction}}~{\ref{item:thm:main:introduction:1}}] 
If $g\in H^\infty$, then $S_g,T_g\in\BB(A^p_{\alpha})$, by Theorem~{\ref{thm:TSM-bounded}~{\ref{item:thm:TSM-bounded:1})}} and Proposition~{\ref{prop:SM-bounded-compact}~{\ref{item:prop:SM-bounded-compact:1}}}, and so $L_g\in\BB(A^p_{\alpha})$. Conversely, assume that  $L_g\in \BB(A^p_{\alpha})$ and apply Proposition~{\ref{prop:norm:dilations}} to conclude that for $r\in (0,1)$, we have  $L_{g_r}\in \BB(A^p_{\alpha})$ with $\|L_{g_r}\|_{\al,p}\le  \|L_{g}\|_{\al,p}$. From 
\[
L_{g_r}=\sum_{k=0}^nS_{g_r}^kT_{g_r}P_k(T_{g_r})+S_{g_r}P_{n+1}(S_{g_r})+g_r(0)P_{n+2}(g_r-g_r(0))\,\delta_0,
\]
we see that for fixed $r\in (0,1)$, all operators on the right are compact, except 
\[
S_{g_r}P_{n+1}(S_{g_r}) = S_{g_rP_{n+1}(g_r)}= M_{g_rP_{n+1}(g_r)}-T_{g_rP_{n+1}(g_r)}-g_rP_{n+1}(g_r)(0)\delta_0.
\]
By Theorem~{\ref{thm:TSM-bounded}~\ref{item:thm:TSM-bounded:2}} we conclude that
\[
L_{g_r}=M_{g_rP_{n+1}(g_r)}+K,
\]
where $K\in\KK(A^p_{\alpha})$ is compact.

Now, for any $\lambda\in\D$, we consider the functions
\[
h_\lambda(z)=\frac{(1-|\lambda|^2)^{\frac{\al+2}{p}}}{(1-\overline{\lambda}z)^{\frac{2\al+4}{p}}}\qquad(z\in \D).
\]
Since $(1-\overline{\lambda}z)^{-\alpha-2}$, $\lambda, z\in\D$, is the Bergman kernel for $A^2_{\alpha}$, 
$\|h_{\lambda}\|_{\alpha,p}=1$, for any $\lambda\in\D$.
(Note that for $\alpha=-1$ the corresponding Bergman kernel is the classical Cauchy kernel.)
 Moreover, it is clear that, for any $\zeta\in\T$,
 $h_{\lambda}\to0$, as $\lambda\to\zeta$, uniformly on compacta. 
So, by Lemma~{\ref{lem:compactness}}, $\|Kh_{\lambda}\|_{\alpha,p}\to 0$.
On the other hand, note that if $G_r=g_rP_{n+1}(g_r)$ then
\[
\|M_{G_r} h_\lambda\|_{\alpha,p}^p
=(\alpha+1)\int_{\D}B_{\alpha}(z,\lambda)\,|G_r(z)|^p(1-|z|^2)^{\alpha}\,\,dA(z)
\quad(\lambda\in\D,\,\alpha>-1),
\]
and 
\[
\|M_{G_r} h_\lambda\|_{-1,p}^p
=\int_{\T}P(\zeta,\lambda)\,|G_r(\zeta)|^p\,\frac{|d\zeta|}{2\pi}
\quad(\lambda\in\D),
\]
where ${\displaystyle B_{\alpha}(z,\lambda)=\frac{(1-|\lambda|^2)^{\alpha+2}}{|1-\overline{\lambda}z|^{2\alpha+4}}}$
and ${\displaystyle P(\zeta,\lambda)=\frac{1-|\lambda|^2}{|1-\overline{\lambda}\zeta|^2}}$
are the Poisson-Bergman (or Berezin) kernel and the classical Poisson kernel, respectively. Thus, since $|G_r|=|g_rP_{n+1}(g_r)|\in C(\overline{\D})$, we have that (see, for instance, \cite[Prop.\ 8.2.7]{Krantz}) 
\[
\lim_{\lambda\to\zeta}\| M_{g_rP_{n+1}(g_r)} h_\lambda\|_{\al,p}^p=|g_rP_{n+1}(g_r)(\zeta)|^p\qquad (\zeta\in \T).
\]
 Hence  
\[
|g_rP_{n+1}(g_r)(\zeta)|^p=\lim_{\lambda\to\zeta}\|L_{g_r}h_\lambda\|^p_{\al,p}\le \|L_g\|^p_{\al,p}\lim_{\lambda\to\zeta}\| h_\lambda\|^p_{\al,p}= \|L_g\|^p_{\al,p},
\]
 for all $\zeta\in \T$ and $0<r<1$, which implies that $gP_{n+1}(g)\in H^\infty$, and so $g\in H^{\infty}$, by Lemma~{\ref{lem:P(g):bounded}}. Thus the proof is finished.
\end{proof}

\begin{proof}[\bf Proof of Theorem~{\ref{thm:main:introduction}}~{\ref{item:thm:main:introduction:2}}] 
First of all observe that if $L\in\BB(A_\al^p)$ then Theorem~{\ref{thm:algebraimpliesginBloch:introduction}
\ref{item:thm:algebraimpliesginBloch:introduction:1}} gives that
$g^k\in A^p_{\alpha}$ for any $k\in\mathbb{N}$, so
$g(0)P_{n+1}(g_0)\,\delta_0\in\BB(A_\al^p)$ and therefore
\[
\sum_{k=0}^nS_g^kT_gP_k(T_g)=L-g(0)P_{n+1}(g_0)\,\delta_0\in\BB(A_\al^p).
\]
Moreover, if either $g^{n+1}\in\B$, if $\alpha>-1$, or $g^{n+1}\in BMOA$, if $\alpha=-1$, then Proposition~{\ref{prop:Bn:nested}} shows that $g\in A^p_{\alpha}$, and we deduce that
$g(0)P_{n+1}(g_0)\,\delta_0\in\BB(A_\al^p)$.
 
Thus, without loss generality, from now on we assume that
\[
L=\sum_{k=0}^nS_g^kT_gP_k(T_g).
\]
 If  $g^n\in\B$ when $\al>-1$, or $g^n\in BMOA$ when $\al=-1$, then, by Proposition~{\ref{prop:Bn:nested}}, the same holds for $g^k$, $1\le k\le n$, and all the operators involved in the definition of $L$ are bounded, hence so is $L$. 

Conversely, if $L=L_g$ is bounded, then by Theorem~{\ref{thm:algebraimpliesginBloch:introduction} \ref{item:thm:algebraimpliesginBloch:introduction:1}}
we have that $T_g$ is bounded. Now, by applying  Proposition~{\ref{prop:norm:dilations}},   Proposition~{\ref{prop:Bn:nested}} and Theorem~{\ref{thm:TSM-bounded}}, for every $0<r<1$, 
we obtain
\begin{align*}
\|T_{g^{n+1}_r}\|
&\le c\left (\|L_{g_r}\|+ \sum_{k=1}^{n-1}\|P_k(T_{g_r})\|\|T_{g_r^{k+1}}\| \right)\\
& \le c'(g,n,p)\left( \|L_g\|+ \sum_{k=1}^{n-1}\|T_{g_r^{n+1}}\|^{\frac{k+1}{n+1}}\right),
\end{align*}
where $0<c'=c'(g,n,p)<\infty$ because $T_g$ is bounded.
Then Remark~{\ref{rem:real-variable:boundedness}} 
and Proposition~{\ref{prop:norm:dilations}} complete the proof.
\end{proof}

We will use  Proposition \ref{prop:Bn:nested} and the following result in the proof of 
Theorem~{\ref{thm:main:introduction}}~{\ref{item:thm:main:introduction:3}}.

\begin{lemma}\label{lem:Tg-pointwise-est}
Let $g\in\B$, and, for any $\lambda\in\D\setminus\{0\}$ and $\gamma>0$, let  
\[
f_{\gamma,\lambda}(z)=\frac{z}{(1-\overline{\lambda}z)^\gamma}
\qquad(z\in\D).
\]
Then:
\begin{enumerate}[label={\sf\alph*)},ref={\sf\alph*)},topsep=3pt, leftmargin=*,itemsep=3pt] 
\item \label{item:lem:Tg-pointwise-est:1}
For any $k\in\N$ and $t\in[0,1]$, we have that
\[
|T_g^kf_{\gamma,\lambda}(t\lambda)|
\le \frac{\|g\|^k_\B}
    {|\lambda|^{k}\gamma^k(1-t|\lambda|^2)^\gamma}.
\] 
\item \label{item:lem:Tg-pointwise-est:2}
If $a_0,\dots,a_n\in\C$ and $\ga|\la|>\|g\|_\B$, then   \[
\left|\sum_{k=0}^n a_k T_g^kf_{\gamma,\lambda}(\lambda)\right|\le |a_0|\frac{|\la|}{(1-|\la|^2)^\gamma} +\left(\sum_{k=1}^n|a_k|\right)\frac{\|g\|_\B}{|\la|\gamma(1-|\la|^2)^\gamma}.
\]
\end{enumerate}
\end{lemma}
\begin{proof} First we prove~{\ref{item:lem:Tg-pointwise-est:1}} by induction on $k$. If $k=1$ and $s\in [0,1]$, we use the estimates
\begin{gather*}
|f_{\gamma,\lambda}(s\lambda)|=\frac{s|\lambda|}{(1-s|\lambda|^2)^\gamma}\le \frac1{(1-s|\la|^2)^\gamma} \\
|g'(s\lambda)|\le \frac{\|g\|_\B}{1-s^2|\lambda|^2}\le 
 \frac{\|g\|_\B}{1-s|\la|^2}
 \end{gather*}
to conclude that 
\begin{align*}
|T_gf_{\gamma,\lambda}(t\lambda)|
&\le t|\lambda|\int_0^1|f_{\gamma,\lambda}(st\lambda)||g'(st\lambda)|ds\\
&\le \|g\|_\B \int_0^1\frac{t|\lambda|\,ds}{(1-st|\lambda|^2)^{\gamma+1}}
\le \frac{\|g\|_{\B}}{|\lambda|\gamma(1-t|\lambda|^2)^\gamma}.
\end{align*}
If the statement holds for some $k\ge 1$ and all $t\in [0,1]$, then, as above, 
\begin{align*}
|T^{k+1}_gf_{\gamma,\lambda}(t\la)|
&\le
 t|\la|\int_0^1|T_g^kf_{\gamma,\lambda}(st\lambda)||g'(st\lambda)|ds\\
&\le\frac{\|g\|^{k+1}_\B}{|\lambda|^k\gamma^k} \int_0^1\frac{t|\lambda|\,ds}{(1-st|\lambda|^2)^{\gamma+1}}\le 
\frac{\|g\|^{k+1}_{\B}}{|\lambda|^{k+1}\gamma^{k+1} (1-t|\la|^2)^\gamma},
\end{align*}
and the result follows.
Finally, {\ref{item:lem:Tg-pointwise-est:2}}~is a straightforward application of~{\ref{item:lem:Tg-pointwise-est:1}}.
\end{proof}

\begin{proof}[\bf Proof of Theorem~{\ref{thm:main:introduction}}~{\ref{item:thm:main:introduction:3}}]  
If $g^{n+1}\in\B $, then $g^{k}\in\B$, for $1\le k\le n+1$,
by Proposition~{\ref{prop:Bn:nested}}, and  the boundedness of $L_g$ follows from the identity $S_g^kT_g=\frac1{k+1}T_{g^{k+1}}$ and Theorem~{\ref{thm:TSM-bounded}~\ref{item:thm:TSM-bounded:1}}.

 Conversely, assume without loss of generality that $P_n(0)=n+1$. If $L_g\in\BB(A^p_{\alpha})$ then  $g\in \B$,  by Theorem~{\ref{thm:algebraimpliesginBloch:introduction}~\ref{item:thm:algebraimpliesginBloch:introduction:1}}. Moreover, Proposition~{\ref{prop:norm:dilations}} shows that  $L_{g_r}\in \BB(A^p_{\alpha})$ and $\|L_{g_r}\|_{\al,p}\le  \|L_{g}\|_{\al,p}$, for any $r\in (0,1)$. Now we write $P_n(z)=(n+1)(1+zQ_n(z))$, and using again the identity  $S_g^kT_g=\frac1{k+1}T_{g^{k+1}}$ we obtain that
\[
L_{g_r}=T_{g_r^{n+1}}+T_{g_r^{n+1}}T_{g_r}Q_n(T_{g_r}) +\sum_{k=0}^{n-1}\tfrac1{k+1}T_{g_r^{k+1}}P_{k}(T_{g_r})+ g_r(0)Q(g_r-g_r(0))\,\delta_0.
\]
For $\gamma>\frac{\alpha+2}p$ and 
$\lambda\in\D\setminus\{0\}$, we apply $L_{g_r}$ to the 
function $f_{\gamma,\lambda}$ from
Lemma~\ref{lem:Tg-pointwise-est}. 
Since $\delta_0(f_{\gamma,\lambda})=0$, we obtain that 
\[
L_{g_r}f_{\gamma,\lambda}
= T_{g_r^{n+1}}f_{\gamma,\lambda}
 +T_{g_r^{n+1}}T_{g_r}Q_n(T_{g_r})f_{\gamma,\lambda} +\sum_{k=0}^{n-1}\tfrac1{k+1}T_{g_r^{k+1}}P_{k}(T_{g_r})
   f_{\gamma,\lambda}.
\]
Now we use the standard estimates 
\[
|h'(\la)|\le\frac{ c_{\alpha}\|h\|_{\alpha,p}}{(1-|\lambda|^2)^{\frac{\alpha+2}{p}+1}},\quad \|f_{\gamma,\lambda}\|_{\alpha,p}\le \frac{c_{\alpha,\gamma}}{(1-|\la|^2)^{\gamma-\frac{\alpha+2}{p}}},
\]  
where $c_{\alpha}>0$ and $c_{\alpha,\gamma}>0$ are constants which depend only on $\alpha$, and $\alpha$ and $\gamma$, respectively,
and Proposition~\ref{prop:norm:dilations}
  to infer that
\[
|(L_{g_r}f_{\ga,\la})'(\la)|
\le \frac{c_{\al}c_{\alpha,\gamma}\|L_g\|_{\al,p}}{(1-|\la|^2)^{\ga+1}}.
\]
Since  
\begin{align*}
(L_{g_r}f_{\ga,\la})'(\la)
&= (g_r^{n+1})'(\la)\,f_{\gamma,\alpha}(\lambda)
    +(g_r^{n+1})'(\la)\,[T_{g_r}Q_n(T_{g_r})
                               f_{\gamma,\alpha}](\lambda)\\
&\quad+\sum_{k=0}^{n-1}\frac{(g_r^{k+1})'(\la)}{k+1}\,
  [P_{k}(T_{g_r})f_{\gamma,\lambda}](\lambda), 
\end{align*}
 by the triangle inequality we have
\begin{align}\label{eqn:basic:Bloch:estimate:0} 
\frac{|\la||(g_r^{n+1})'(\la)|}{(1-|\la|^2)^\ga}
&\le |(g_r^{n+1})'(\la)||[T_{g_r}Q_n(T_{g_r})f_{\gamma,\lambda}](\la)|+\frac{c_{\al}c_{\alpha,\gamma}\|L_g\|_{\al,p}}{(1-|\la|^2)^{\ga+1}}\\
&\nonumber\,\,\, +
\sum_{k=0}^{n-1}\frac{|(g_r^{k+1})'(\la)|}{k+1}|[P_{k}(T_{g_r})f_{\gamma,\lambda}](\la)|.
\end{align}
We want to estimate the terms on the right with the help of Lemma~{\ref{lem:Tg-pointwise-est}~\ref{item:lem:Tg-pointwise-est:2}}. To this end, note  that  $n$, $Q_n$, and $P_k$, for $0\le k\le n-1$,  depend only on $L_g$, so there exists a constant $c=c(L_g)>0$ depending only on  $L_g$ such that, for $\gamma|\lambda|>\|g\|_\B$, we have that  
\[
|[T_{g_r}Q_n(T_{g_r})f_{\gamma,\lambda}](\la)|\le c\,\frac{\|g\|_\B}{|\lambda|\gamma(1-|\la|^2)^\gamma},
\]
and 
\[
|[P_{k}(T_{g_r})f_{\gamma,\lambda}](\lambda)|\le c\left(\frac{|\lambda|}{(1-|\lambda|^2)^\gamma}+\frac{\|g\|_\B}{|\lambda|\gamma(1-|\lambda|^2)^\gamma}\right)
\quad(0\le k\le n-1).
\]
Using these inequalities in~{\eqref{eqn:basic:Bloch:estimate:0}} we obtain
\begin{align}\label{eqn:basic:Bloch:estimate:1} |(g_r^{n+1})'(\lambda)|
&\le |(g_r^{n+1})'(\lambda)|\frac{c\|g\|_\B}{|\lambda|^2\gamma}+\frac{c_{\al}c_{\alpha,\gamma}\|L_g\|_{\alpha,p}}{|\la|(1-|\lambda|^2)}   \\
&\nonumber\,\,\, +
c\sum_{k=0}^{n-1}\frac{|(g_r^{k+1})'(\lambda)|}{k+1}
\biggl(1+\frac{\|g\|_\B}{|\la|^2\gamma}\biggr),
\end{align}
when $\gamma|\lambda|>\|g\|_\B$.  
Now if $\gamma$ satisfies $\gamma>8(c+1)\|g\|_\B$, \eqref{eqn:basic:Bloch:estimate:1}  gives for $|\lambda|>\frac1{2}$ 
\[
\frac1{2}|(g_r^{n+1})'(\lambda)|\le
 \frac{2c_{\alpha}c_{\alpha,\gamma}\|L_g\|_{\alpha,p}}{(1-|\lambda|^2)}+\frac{3c}{2}\sum_{k=0}^{n-1}\frac{|(g_r^{k+1})'(\lambda)|}{k+1}.
\]
Thus, we either have $\|g^{n+1}_r\|_\B=\sup_{|\lambda|\le\frac1{2}}(1-|\lambda|^2)|(g_r^{n+1})'(\lambda)|$, or, by  Proposition \ref{prop:Bn:nested} and the last inequality,
\[
\|g^{n+1}_r\|_\B\le
4c_{\alpha}c_{\alpha,\gamma}\|L_g\|_{\alpha,p}+
 c'\sum_{k=0}^{n-1}\frac1{k+1}\|g_r^{n+1}\|_\B^{\frac{k+1}{n+1}}. 
\]
This shows that  $\|g^{n+1}_r\|_\B$  stays bounded when $r\to 1^-$, {\em  i.e.} $g^{n+1}\in \B$.
\end{proof}


\subsection{Compositions of two analytic paraproducts. }\label{2letter-words}   Corollary \ref{prop:boundedness:two:letters:operators} together with the identities
\begin{align*}
M^2_g&=S_g^2+2S_gT_g+g^2(0)\,\delta_0\\
M_gT_g&= S_gT_g+T^2_g\\
S_gM_g &= S_gT_g+ S^2_g\\
 T_gM_g&=S_gT_g\\
 T_gS_g&=S_gT_g-T^2_g-g(0)(g-g(0))\,\delta_0\\
 M_gS_g&=S_gT_g-T^2_g+S^2_g-g(0)(g-g(0))\delta_0
\end{align*}  
yield a complete characterization of compositions of two analytic paraproducts.   A summary for $\alpha>-1$
 is provided in the following table. The analogue for the $H^p$-case can be obtained replacing $\B$ by $BMOA$.

\noindent
\hspace*{\fill}
\hspace*{\fill}
\vspace{-4pt}
\begin{center}
\scalebox{1.06}[1.1]{
\begin{tabular}
{||c|c|c|c||}\hline\hline
\multicolumn{4}{||c||}{{\footnotesize\bf Boundedness of composition of analytic paraproducts on $A_\al^p$, $\alpha>-1$}}\\
\hline\hline
&{\footnotesize $T_g$} 
&{\footnotesize $S_g$}
&{\footnotesize $M_g$} \\
\hline\hline
{\footnotesize $T_g$ }
&{\footnotesize $T_g^2\in\BB(A_{\alpha}^p)\Leftrightarrow T_g\in\BB(A_{\alpha}^p)$} 
&{\footnotesize $S_gT_g\in\BB(A_{\alpha}^p)\Leftrightarrow   T_{g^2}\in\BB(A_{\alpha}^p)$}  
&{\footnotesize $M_gT_g\in\BB(A_{\alpha}^p)\Leftrightarrow    
           T_{g^2}\in\BB(A_{\alpha}^p)$} \\
{\footnotesize $\mbox{ }$ }
&{\footnotesize \hspace*{32pt}$\Leftrightarrow g\in\B$} 
&{\footnotesize \hspace*{42pt}$\Leftrightarrow   g^2\in\B$}  
&{\footnotesize \hspace*{46pt}$\Leftrightarrow    
           g^2\in\B$} \\           
\hline
{\footnotesize $S_g$} 
&{\footnotesize $T_gS_g\in\BB(A_{\alpha}^p)\Leftrightarrow    
           T_{g^2}\in\BB(A^p_{\alpha})$}  
&{\footnotesize $S_g^2\in\BB(A_{\alpha}^p)\Leftrightarrow    
           S_g\in\BB(A^p_{\alpha})$} 
&{\footnotesize $M_gS_g\in\BB(A_{\alpha}^p)\Leftrightarrow    
           S_g\in\BB(A^p_{\alpha})$} \\ 
{\footnotesize $\mbox{ }$ }
&{\footnotesize \hspace*{42pt}$\Leftrightarrow    
          g^2\in\B$}  
&{\footnotesize \hspace*{32pt}$\Leftrightarrow    
           g\in H^{\infty}$} 
&{\footnotesize \hspace*{46pt}$\Leftrightarrow    
           g\in H^{\infty}$} \\            
\hline
{\footnotesize $M_g$} 
&{\footnotesize $T_gM_g\in\BB(A_{\alpha}^p)\Leftrightarrow    
           T_{g^2}\in\BB(A_{\alpha}^p)$} 
&{\footnotesize $S_gM_g\in\BB(A_{\alpha}^p)\Leftrightarrow    
           S_g\in\BB(A_{\alpha}^p)$}
&{\footnotesize $M_g^2\in\BB(A_{\alpha}^p)\Leftrightarrow    
           M_g\in\BB(A_{\alpha}^p)$} \\  
{\footnotesize $\mbox{ }$ }
&{\footnotesize \hspace*{46pt}$\Leftrightarrow    
           g^2\in\B$} 
&{\footnotesize \hspace*{46pt}$\Leftrightarrow    
           g\in H^{\infty}$}
&{\footnotesize  \hspace*{32pt}$\Leftrightarrow    
           g\in H^{\infty}$} \\            
\hline \hline                 
\end{tabular} }
\end{center}
\vspace{8pt}

\section{Proof of Theorem~{\ref{thm:counterexamples}}}

The following proposition is strongly used in 
the proof of Theorem~{\ref{thm:counterexamples}}.
\begin{proposition}\label{prop:main:prop:counterexamples}
Let $g\in\H(\D)$. Assume that $g$ is bounded away from zero, that is, $\inf_{z\in\D}|g(z)|>0$. 
Let $h$ be a branch of the logarithm of $g$, and, for any $\beta\in\R$, define the $\beta$-power of $g$ as $g^{\beta}:=e^{\beta h}$.  
Then:
\begin{enumerate}[label={\sf\alph*)},topsep=3pt, leftmargin=*,itemsep=3pt] 
\item \label{item:prop:main:prop:counterexamples:1}
If $g\in BMOA$ ($g\in VMOA$), then $g^{\beta}\in BMOA$ ($g^{\beta}\in VMOA$, resp.), for any $\beta<1$.
\item \label{item:prop:main:prop:counterexamples:2} 
If $g\in BMOA$ ($g\in VMOA$), then
 $S_{g^{\beta}}T^2_{g^{\beta}}\in\BB(A^p_{\alpha})$
 ($S_{g^{\beta}}T^2_{g^{\beta}}\in\KK(A^p_{\alpha})$, resp.), for any $\alpha\ge-1$,  $\beta\in(0,\frac23)$, and $p>0$.   
\end{enumerate}
\end{proposition}

A key tool in the proof of Proposition~{\ref{prop:main:prop:counterexamples}} is the following simple computational lemma.

\begin{lemma}\label{lem:main:prop:counterexamples}
Let $g\in\H(\D)$ be a zero free function, and, for any $\beta\in\R$, let $g^{\beta}$ be as in the statement of the preceding proposition. 
Then 
\begin{equation}\label{eqn:lem:main:prop:counterexamples}
S_{g^{\beta}}T^2_{g^{\beta}}
=\tfrac{(2\beta-1)\beta}{1-\varepsilon}\, T_gT_{g^{1-\varepsilon}}M_{g^{2\beta-2+\varepsilon}}T_{g^{\beta}}
+\tfrac{\beta^2}{1-\varepsilon}\,T_gT_{g^{1-\varepsilon}}M_{g^{3\beta-2+\varepsilon}},
\end{equation}
for every $\beta\in\R$ and  $\varepsilon\in\R\setminus\{1\}$. 
\end{lemma}

\begin{proof}
The fact that $(g^{\beta})^2=g^{2\beta}$ gives that 
$L:=S_{g^{\beta}}T^2_{g^{\beta}}
=\tfrac12 T_{g^{2\beta}}T_{g^{\beta}}$. 
Thus, for any $f\in\H(\D)$, we have that
\[
(Lf)'
=\tfrac12\,(g^{2\beta})'\,T_{g^{\beta}}f
=\beta\,g'\,F,\quad\mbox{
where $F=g^{2\beta-1}\,T_{g^{\beta}}f$.}
\]
Since $Lf(0)=0$, it follows that $Lf=\beta\,T_gF$.
 Now $F(0)=0$ and 
\begin{align*}
F'
&= (2\beta-1)\,g^{2\beta-2}\,g'\,T_{g^{\beta}}f
    +\beta\,g^{3\beta-2}\,g'\,f \\
&= (g^{1-\varepsilon})'\,\bigl(\tfrac{2\beta-1}{1-\varepsilon}\,
 g^{2\beta-2+\varepsilon}\,T_{g^{\beta}}f
 + \tfrac{\beta}{1-\varepsilon}\,
   g^{3\beta-2+\varepsilon}\,f\bigr)  ,
\end{align*}
for any $\varepsilon\in\R\setminus\{1\}$. Therefore
\[
F=\tfrac{2\beta-1}{1-\varepsilon}\,T_{g^{1-\varepsilon}}\,
  M_{g^{2\beta-2+\varepsilon}}\,T_{g^{\beta}}f
   +\tfrac{\beta}{1-\varepsilon}\,T_{g^{1-\varepsilon}}\,
  M_{g^{3\beta-2+\varepsilon}}f,
\]
and hence~{\eqref{eqn:lem:main:prop:counterexamples}} holds.
\end{proof}

\begin{proof}[\bf Proof of Proposition~{\ref{prop:main:prop:counterexamples}}]$\mbox{ }$
\begin{enumerate}[label={\sffamily{\alph*)}},topsep=3pt, 
leftmargin=0pt, itemsep=4pt, wide, listparindent=0pt, itemindent=6pt] 
\item Just observe that,  since $g$ is bounded away from zero, $g^{\beta-1}$ is bounded for $\beta<1$, and so we have the estimate $(1-|z|^2)|(g^{\beta})'|^2\lesssim (1-|z|^2)|g'(z)|^2$. 
\item Let $g\in BMOA$ ($g\in VMOA$), $\alpha\ge-1$,  
$\beta\in(0,\frac23)$, and $p>0$. Since $\beta<\frac23$, there is  $\varepsilon\in\R$ such that $0<\varepsilon<\min(2-3\beta,1)$. Taking into account that
$\beta>0$, it follows that $\varepsilon\in(0,1)$ and
$2\beta-2+\varepsilon<3\beta-2+\varepsilon<0$. As a consequence,
we have that:
\begin{itemize}
\item $T_g,T_{g^{1-\varepsilon}}, T_{g^{\beta}}\in\BB(A^p_{\alpha})$ ($T_g,T_{g^{1-\varepsilon}}, T_{g^{\beta}}\in\KK(A^p_{\alpha})$, resp.), by~{\ref{item:prop:main:prop:counterexamples:1}}.
\item $M_{g^{2\beta-2+\varepsilon}},M_{g^{3\beta-2+\varepsilon}}\in\BB(A^p_{\alpha})$, since $g^{2\beta-2+\varepsilon},g^{3\beta-2+\varepsilon}\in H^{\infty}$, because $g$ is bounded away from zero.
\end{itemize}
Moreover, since $\varepsilon<1$, {\eqref{eqn:lem:main:prop:counterexamples}}~holds, and we conclude that 
 $S_{g^{\beta}}T^2_{g^{\beta}}\in\BB(A_{\alpha})$
 ($S_{g^{\beta}}T^2_{g^{\beta}}\in\KK(A^p_{\alpha})$, resp.). Hence the proof is complete.\qedhere 
\end{enumerate}
\end{proof}

 We also need  the following  auxiliary result.

\begin{lemma}\label{lem:vanishing:Carleson:measure:improved}
Let $f\in C(\overline{\D}\setminus\{1\})$ such that
\begin{equation}
\label{eqn:lem:vanishing:Carleson:measure:improved:0}
\lim_{\substack{z\to 1\\z\in\D}} (1-z)\,f(z)=0.
\end{equation}
Then $d\mu(z)=(1-|z|^2)|f(z)|^2\,dA(z)$ is a vanishing Carleson measure for 
$H^p$, $0<p<\infty$.
\end{lemma}

\begin{proof}
Let $\Omega_{\delta}=\D\cap D(1,\delta)$, for every $0<\delta<1$. Let $a\in\D$ and $0<\delta<1$. Then 
\[
\int_{\D}\frac{1-|a|^2}{|1-\overline{a}z|^2}\,d\mu(z)
=\biggl\{\int_{\D\setminus\Omega_{\delta}}+\int_{\Omega_{\delta}}\biggr\}(1-|\phi_a(z)|^2)|f(z)|^2\,dA(z)
=I_{\delta}+J_{\delta}.
\]
Now, by \cite[Proposition 1.4.10]{Rudin}, we have that
\begin{align}
\label{eqn:lem:vanishing:Carleson:measure:improved:1}
I_{\delta}
&\le \biggl(\sup_{z\in\D\setminus\Omega_{\delta}} |f(z)|^2\biggr)\int_{\D} (1-|\phi_a(z)|^2)\,dA(z) \\
\nonumber
&\le C_1\,(1-|a|^2)\sup_{z\in\D\setminus\Omega_{\delta}} |f(z)|^2,
\end{align}
where $C_1>0$ is an absolute constant. Next recall that, since $\log(1-z)$ is a function in $BMOA$ (where $\log$ denotes the principal branch of the logarithm), $d\mu_1(z)=\dfrac{1-|z|^2}{|1-z|^2}\,dA(z)$ is a Carleson measure for the Hardy spaces, and so
\begin{align}
\label{eqn:lem:vanishing:Carleson:measure:improved:2}
 J_{\delta}&\le  
 \biggl(\sup_{z\in\Omega_{\delta}}
 |1-z| |f(z)|\biggr)^2
 \int_{\D}\frac{1-|a|^2}{|1-\overline{a}z|^2}\,d\mu_1(z)\\
 \nonumber
 &\le  
 C_2\biggl(\sup_{z\in\Omega_{\delta}}
 |1-z| |f(z)|\biggr)^2,
\end{align}
where $C_2>0$ is an absolute constant.
Since  $f\in C(\overline{\D}\setminus\{1\})$, it is clear that \eqref{eqn:lem:vanishing:Carleson:measure:improved:0}, 
\eqref{eqn:lem:vanishing:Carleson:measure:improved:1}, and 
\eqref{eqn:lem:vanishing:Carleson:measure:improved:2} imply that $\mu$ is a vanishing Carleson measure for the Hardy spaces, {\em i.e.}
\[
\lim_{|a|\to1^{-}}\int_{\D}\frac{1-|a|^2}{|1-\overline{a}z|^2}\,d\mu(z)=0.\qedhere
\]
\end{proof}

\begin{proof}[{\bf Proof of Theorem~{\ref{thm:counterexamples}}}]
$\mbox{ }$
\begin{enumerate}[label={\sffamily{\alph*)}},topsep=3pt, 
leftmargin=0pt, itemsep=4pt, wide, listparindent=0pt, itemindent=6pt] 
\item Assume the contrary, {\em i.e.}\ $S_gT_g^2\in \BB(A_\alpha^p)$, for some $\alpha\ge -1,~p>0$.  Then a standard estimate yields for every $f\in A_\alpha^p$,
\begin{equation}\label{unbounded-ex}
|(S_gT_g^2f)'(r)| \lesssim (1-r)^{-1-\frac{\alpha+2}{p}}\|S_gT_g^2\|_{\alpha,p}\|f\|_{\alpha,p}
\quad(r\in [0,1)).
\end{equation}
The usual test functions $f_{r,k}(z)=(1-rz)^{-k}$, for $z\in\D$, with $r\in (0,1)$,  $kp>\alpha+2$, satisfy $\|f_{r,k}\|_{\alpha,p}\simeq (1-r)^{-k+\frac{\alpha+2}{p}}$, and  
\begin{align*} |(S_gT_g^2f_{r,k})'(r)|&=\frac1{1-r}\log\frac{e}{1-r}\int_0^r\frac{ds}{(1-s)(1-rs)^k}\\&
\ge \frac1{1-r}\log\frac{e}{1-r}\int_0^r\frac{rds}{(1-rs)^{k+1}}\\&=\frac1{k(1-r)}\log\frac{e}{1-r}\left(\frac1{(1-r^2)^k}-1\right),
\end{align*}
which contadicts \eqref{unbounded-ex} when $r\to 1^-$. 
\item If $f\in\B$, then
\[
|f(z)|\lesssim\log\biggl(\frac{e}{1-|z|}\biggr)\qquad(z\in\D),
\]
but
\[
\lim_{r\to1^{-}}\frac{|g^{2\beta}(r)|}{\log\bigl(\frac{e}{1-r}\bigr)}
=\lim_{r\to1^{-}}\biggl[\log\biggl(\frac{e}{1-r}\biggr)\biggr]^{2\beta-1}
=\infty\qquad(\beta>\tfrac12),
\]
and so $g^{2\beta}\notin\B$, for any $\beta>\frac12$. 

Now let us prove that $S_{g^{\beta}}T^2_{g^{\beta}}\in\KK(A^p_{\alpha})$, for any $\alpha\ge-1$ and $p>0$.
We know that $g\in BMOA$. 
Moreover, since $z\mapsto \frac{e}{1-z}$ maps the half-disc
\[
D^{-}=\{z\in\C:|z-1|<1+\tfrac{e}{2},\,\Re z<1\}
\]
onto the domain 
$\{z\in\C:|z|>\tfrac{2e}{2+e},\,\Re z>0\}$,
 it follows that $g$ is 
bounded away from zero and $g^{\beta}$ extends analytically to 
$D^{-}$.
In particular, 
$(g^{\beta})'\in C(\overline{\D}\setminus\{1\})$ and satisfies
\[
\lim_{\substack{z\to1\\z\in\D}}(1-z)(g^{\beta})'(z)
=\beta\lim_{\substack{z\to1\\z\in\D}}g^{\beta-1}(z)=0
\qquad(\beta<1).
\]
Then Lemma~{\ref{lem:vanishing:Carleson:measure:improved}} gives that $g^{\beta}\in VMOA$, for every $\beta<1$, and 
 Proposition~{\ref{prop:main:prop:counterexamples}}
shows that $S_{g^{\beta}}T^2_{g^{\beta}}\in\KK(A^p_{\alpha})$, for any $\alpha\ge-1$, $\beta\in(0,\frac23)$, and $p>0$.
\qedhere 
\end{enumerate}
\end{proof}

\end{document}